\documentclass[a4paper,11pt]{article}
\usepackage[T1]{fontenc}
\usepackage[utf8]{inputenc}
\usepackage[english]{babel}
\usepackage[stretch=10]{microtype}
\usepackage{mathtools}
\mathtoolsset{centercolon}
\usepackage{amssymb}
\usepackage{bbm}
\usepackage{bm}
\usepackage[shortlabels]{enumitem}
\usepackage{comment}
\usepackage{xcolor}
\usepackage{graphicx}
\usepackage{hyperref}%importante l'ordine hyperref --> amsthm&(geometry-->todonotes) --> cleveref
\hypersetup{
	colorlinks,
	linkcolor={red!50!black},
	citecolor={blue!50!black},
	urlcolor={black}
}
\usepackage{amsthm}
\usepackage{geometry}
\usepackage{cleveref}
\DeclarePairedDelimiter{\abs}{\lvert}{\rvert}
\DeclarePairedDelimiter\norm{\lVert}{\rVert}
\DeclarePairedDelimiter{\fences}{(}{)}
\newcommand{\V}{V_n\fences}
\renewcommand{\Pr}{\operatorname{\mathbb{P}}\fences}
\DeclareMathOperator{\Ex}{\mathbb{E}}
\newcommand{\Exval}{\Ex\fences}
\newcommand{\ls}{\leqslant}
\newcommand{\gr}{\geqslant}
\renewcommand{\leq}{\leqslant}
\renewcommand{\geq}{\geqslant}
\renewcommand{\le}{\leqslant}
\renewcommand{\ge}{\geqslant}
\newcommand*{\numberset}{\mathbb}
\newcommand*{\N}{\numberset{N}}
\newcommand*{\R}{\numberset{R}}
\newcommand*\de{\mathop{}\!\mathrm{d}}
\newcommand*{\rec}[1]{\frac{1}{#1}}
\newcommand*{\indic}[1]{\mathbbm{1}_{#1}}
\newcommand{\cdf}{\mathrm{F}}
\newcommand{\B}{B}
\renewcommand{\epsilon}{\varepsilon}
\renewcommand{\theta}{\vartheta}
\renewcommand{\phi}{\varphi}
\newcommand{\vball}{\kappa}
\theoremstyle{plain}
\newtheorem{theorem}{Theorem}[section]
\newtheorem{corollary}[theorem]{Corollary}
\newtheorem{lemma}[theorem]{Lemma}
\newtheorem{proposition}[theorem]{Proposition}
\theoremstyle{definition}

\newtheorem{remark}[theorem]{Remark}
\title{\textbf{Threshold phenomena for high-dimensional random polytopes}}
\author{Gilles Bonnet\thanks{Research supported by the Deutsche Forschungsgemeinschaft (DFG) via RTG 2131 \emph{High-Dimensional Phenomena in Probability -- Fluctuations and Discontinuity.}}, Giorgos Chasapis\thanks{Research supported by the National Scholarship Foundation (IKY), sponsored by the act ``Scholarship grants for second-degree graduate studies'', from resources of the operational program ``Manpower Development, Education and Life-long Learning'', 2014-2020, co-funded by the European Social Fund (ESF) and the Greek state.}, Julian Grote\footnotemark[1]\,,\\Daniel Temesvari\footnotemark[1] \,and Nicola Turchi\footnotemark[1]}
\date{}
\makeatletter
\def\blfootnote{\xdef\@thefnmark{}\@footnotetext}
\makeatother

\begin{document}

\maketitle

\begin{abstract}
Let $X_1,\ldots,X_N$, $N>n$, be independent random points in $\R^n$, distributed according to the so-called beta or beta-prime distribution, respectively. We establish threshold phenomena for the volume, intrinsic volumes, or more general measures of the convex hulls of these random point sets, as the space dimension $n$ tends to infinity. The dual setting of polytopes generated by random halfspaces is also investigated.
\end{abstract}

\bigskip
\tableofcontents
\bigskip

\pagebreak

\section{Introduction and main results}

In the last decades, random polytopes have become one of the outstanding models of study in stochastic geometry, connecting problems and methods from classical convexity and probability theory, and bringing forth numerous applications in other fields of mathematics like optimization, extreme value theory, random matrices and algorithmic geometry, to name just a few. Among the vast literature on the subject, we direct the reader to the recent survey papers \cite{Hug,Reitzner} and the references therein for a detailed account on the matter.

A particular issue that has been studied in many aspects is the complexity of volume computation and approximation of high-dimensional convex bodies by random polytopes. In the general setting, one may consider the convex hull $\mathrm{conv}\{X_1,\ldots,X_N\}$ of a finite number of points chosen randomly from the interior of a convex body $K$ in $\R ^n$, and investigate conditions under which this convex hull ``well-approximates'' the original body, for example in terms of the volume or other geometric parameters. In a seminal work, Dyer, F\"{u}redi and McDiarmid \cite{DFMcD} proved that the expected volume of the convex hull $C_N=\mathrm{conv}\{X_1,\ldots,X_N\}$ of $N>n$ points chosen uniformly and independently from the vertices of the $n$-dimensional cube $[-1,1]^n$, exhibits a phase transition when $N$ is taken to be exponential in the dimension $n$, namely, that for every $\varepsilon>0$,
\[
\lim_{n\to\infty} \frac{\Ex  V_n(C_N)}{V_n([-1,1]^n)} = \begin{cases} 0 &\text{ if } N\ls (2e^{-1/2}-\varepsilon)^n \\
1 &\text{ if } N\gr (2e^{-1/2}+\varepsilon)^n,
\end{cases}
\]
where $V_n$ denotes the $n$-dimensional volume of a set. The method introduced in \cite{DFMcD} influenced a number of later works, like for instance the approach that B\'{a}r\'{a}ny and P\'{o}r \cite{BP} used to prove the existence of $\pm 1$ polytopes with a super-exponential number of facets. Subsequently, new volume threshold results were established by Gatzouras and Giannopoulos \cite{GG} for random polytopes generated by a wide class of probability measures $\mu$ in $\R ^n$, as well as Pivovarov \cite{Piv}, who treated the case of independent points with respect to the Gaussian measure in $\R ^n$ and the uniform measure on the Euclidean sphere. Moreover, Pivovarov considered the dual setting of polytopes generated as sections of random halfspaces with respect to the same probability measures. We stress that the authors in both \cite{GG} and \cite{Piv} exploit the method of \cite{DFMcD}, which due to its geometric viewpoint seems to be applicable for a wide variety of probability distributions. 

Let $N$ and $n$ be natural numbers, \(N>n\), and $X_1, X_2,\ldots,X_N$ be independent and identically distributed random points in $\R^n$, equipped with the Euclidean norm $\norm{\cdot}$ and its corresponding unit ball $\B_2^n$. In this text, we draw our attention to the following two probability distribution models.
\begin{enumerate}[(a)]
\item The \textit{Beta model}, with parameter $\beta > -1$: $X_1$ has density proportional to
\begin{equation*}
(1-\norm{x}^2)^\beta,\quad x\in\B_2^n.
\end{equation*}
We are interested in the random polytope given by
\[
P_{N,n}^\beta:= \mathrm{conv}\{X_1,\ldots,X_N\}.
\]
\item The \textit{Beta-prime model}, with parameters $\beta > n/2$ and $\sigma > 0$: $X_1$ has density proportional to
\begin{equation*}
\biggl(1+ \frac{ \norm{x}^2 }{\sigma^2} \biggr)^{-\beta},\quad x\in\R^n.
\end{equation*}
As before, we consider the random polytope
\[
	\tilde{P}_{N,n}^{\beta,\sigma}:= \mathrm{conv}\{X_1,\ldots,X_N\}.
	\] 
\end{enumerate}

Lately, the high-dimensional geometry of sets arising from these models of randomness have been studied extensively; for instance, in terms of properties of their volume \cite{GKT}, facet numbers \cite{BGTTTW} or intrinsic volumes \cite{KTT}. Asymptotic estimates on the expected volume of the polytope $P_{N,n}^\beta$, as $N\to \infty$, were derived by Affentranger \cite{Af} for any fixed dimension $n$ and parameter $\beta$. Note also, that the gnomonic projection of a uniformly distributed point on the half-sphere is beta prime distributed, which is exploited in \cite{BGTTTW} and \cite{KMTT}.

In this article, we prove threshold results for the volumes and intrinsic volumes of $P_{N,n}^\beta$ and the content of $	\tilde{P}_{N,n}^{\beta,\sigma}$ with respect to log-concave isotropic measures, as the space dimension tends to infinity. In particular, it turns out that the polytope $P_{N,n}^\beta$ tends to capture the whole volume of $B_2^n$ only if the number of points $N$ is superexponential in $n$.

\begin{theorem}[Threshold for beta polytopes] \label{Theorem1}
	\label{thm:beta}
	Fix $\epsilon\in(0,1)$ and let $-1 < \beta = \beta(n)$ and $N=N(n)$ be sequences. Then, 
	\[
	\lim_{n\to \infty}\frac{\Ex \V{P_{N,n}^\beta}}{\V{\B_2^n}}  =  
	\begin{cases} 
	0 &\text{ if } N\ls \exp\left((1-\epsilon)(\beta+\frac{n+1}{2})\log n\right) \\ 
	1 &\text{ if } N\gr \exp\left((1+\epsilon)(\beta+\frac{n+1}{2})\log n\right).
	\end{cases}
	\]
\end{theorem}

Although the statement of \Cref{Theorem1} would still hold when replacing the factor \(\beta+(n+1)/2\) by \(\beta+n/2\), we write the former version because the condition on \(\epsilon \) constant can be actually relaxed into \(\epsilon=\epsilon(n)\), where \(\epsilon(n)\to 0^+\) slowly enough. This makes the aforementioned factors not interchangeable. The admissible speed of decay of \(\epsilon(n)\) is addressed in \Cref{remark:general}. 

A special case of \Cref{Theorem1} is of particular interest. By its very definition (see Section \ref{ssec.beta.def} below), the beta distribution for $\beta=0$ coincides with the uniform probability measure on the Euclidean ball $B_2^n$. The following is thus an immediate corollary of Theorem \ref{Theorem1}.

\begin{corollary}
Fix $\epsilon\in(0,1)$ and let $N=N(n)$ be a sequence of positive integers.
Let $X_1,\ldots,X_N$ be independent random points uniformly distributed on $B_2^n$ and set \(B_{N,n}\coloneqq\mathrm{conv}\{X_1,\ldots,X_N\}\).
Then,
	\[
	\lim_{n\to \infty}\frac{\Ex \V{B_{N,n}}}{\V{\B_2^n}} = 
	\begin{cases} 
	0 &\text{ if } N\ls \exp\left((1-\epsilon)(\frac{n+1}{2})\log n\right) \\ 
	1 &\text{ if } N\gr \exp\left((1+\epsilon)(\frac{n+1}{2})\log n\right).
	\end{cases}
	\]
\end{corollary}

Moreover, since the uniform distribution on the unit sphere $S^{n-1}$ arises as the weak limit of the beta distribution, as $\beta\to -1$ (see for example the proof of Theorem 2.7 in \cite{GKT}), the result of Theorem 2.4 in \cite{Piv} can be recovered by Theorem~\ref{Theorem1}.

\begin{corollary} \label{cor.sphere}
Fix $\epsilon\in(0,1)$ and let $N=N(n)$ be a sequence of positive integers.
Let \(X_1,\ldots,X_N\) be independent random points uniformly distributed on \(S^{n-1}\) and set \(S_{N,n}\coloneqq\mathrm{conv}\{X_1,\ldots,X_N\}\).
Then,
        \[
            \lim_{n\to \infty}\frac{\Ex \V{S_{N,n}}}{\V{\B_2^n}} = 
            \begin{cases} 
            0 &\text{ if } N\ls \exp\bigl( (1-\epsilon) \bigl( \frac{n-1}{2} \bigr)\log n \bigr)\\
            1 &\text{ if } N\gr \exp\bigl( (1+\epsilon) \bigl( \frac{n-1}{2} \bigr)\log n \bigr).
            \end{cases}
        \]
\end{corollary}

Similar threshold statements hold also for the intrinsic volumes of $P_{N,n}^\beta$.
Intrinsic volumes are geometric functionals which arise from the computation of the volume of the Minkowski sum of two convex sets in \(\R^n\). Namely, for a convex set \(K\) and \(t>0\), the volume of $K+tB_2^n$ can be written as a polynomial of degree \(n\) in $t$:
\[
V_n(K+tB_2^n) = \sum_{j=0}^n t^{n-j}V_{n-j}(B_2^{n-j})V_j(K),
\]
with non-negative coefficients $(V_j(K))_{j=0}^n$.
The term $V_j(K)$, $j\in \{0,\ldots,n\}$, is called the \emph{$j$-th intrinsic volume} of $K$. 
In particular, $V_n(K)$ is the volume of $K$, $V_{n-1}(K)$ is half of its surface area and $V_1(K)$ is a constant multiple of its mean width, respectively. The intrinsic volumes are of great interest in valuation theory, since they form a basis of the vector space of all continuous motion invariant valuations on the set of convex bodies in $\R^n$. This observation is the content of Hadwiger's characterization theorem, see for instance Lemma 4.2.6 in \cite{SchneiderBuch}.

As pointed out in \cite{KTT}, the expected $k$-th intrinsic volume of $P_{N,n}^\beta$ is directly connected to the expected $k$-dimensional volume of $P_{N,k}^\alpha$ for some different parameter $\alpha$ depending on $\beta$, $k$ and $n$. Because of this, \Cref{Theorem1} can be applied to establish threshold results for the intrinsic volumes $V_k(P_{N,n}^\beta)$, $k\in \{1,\ldots,n\}$, for different regimes of $k=k(n)$. 

On the other hand, the case that $k$ is a fixed integer is of independent interest, since it amounts to studying the threshold behaviour of $V_n(P_{N,n}^\beta)$ as $\beta\to\infty$ while the dimension $n$ stays fixed. We prove the following.
\begin{theorem}[Threshold for intrinsic volumes of beta polytopes] \label{thm.V_k.fixed_k}
Fix $\epsilon\in(0,1)$ and $k\in \mathbb{N}$, and let $-1<\beta=\beta(n)$ and $N=N(n)$ be arbitrary sequences of real and natural numbers, respectively. Then
\[
\lim_{n\to \infty} \frac{\Ex V_k(P_{N,n}^\beta)}{V_k(B_2^n)} = 
\begin{cases} 
1 &\text{ if } N\gr \exp\left(\exp\left( (1+\epsilon)\log\left(\beta+\frac{n-k}{2}\right)\right)\right)\\
0 &\text{ if }  N\ls \exp\left(\exp\left( (1-\epsilon)\log\left(\beta+\frac{n-k}{2}\right)\right)\right) .
\end{cases}
%\begin{cases} 
%1 &\quad\text{if } N\gr \exp\left( \left(\beta+\frac{n-k}{2}\right)(1+\varepsilon)^{\beta+\frac{n-k}{2}} \right) \\
%0 &\quad\text{if }  N\ls \exp\left( \left(\beta+\frac{n-k}{2}\right)(1-\varepsilon)^{\beta+\frac{n-k}{2}} \right).
%\end{cases}
\]
\end{theorem}
The proof of Theorem \ref{thm.V_k.fixed_k}, as well as a general discussion on threshold phenomena for the intrinsic volumes of $P_{N,n}^\beta$ is the content of Section \ref{sssec.intrinsic}.

Next, we treat the case of the beta-prime distribution. Since the underlying measure is not compactly supported, in the spirit of \cite{PivPhd}, we replace the role of the normalized volume on the ball by an arbitrary isotropic log-concave probability measure $\mu$ on $\R ^n$, see Subsection \ref{sec:iso.log.conv.prob.meas} for the definition.
In the sequel we will use the notation $ a \ll b $ if $\frac{a}{b} \to 0 $ as $ n \to \infty $.
\begin{theorem}[Threshold for beta prime polytopes] \label{thm.beta.prime.sigma}
    Fix $\epsilon\in(0,1)$. Let \(\mu = \mu_n \) denote a sequence of isotropic log-concave measures on $\mathbb{R}^n$, let $ \sigma = \sigma(n) > 0 $ and $\beta=\beta(n)$ be sequences of real numbers, and let $N=N(n)$ be a sequence of natural numbers.
    Let \(\beta-\frac{n}{2}\gg\log n\).
    \begin{enumerate}[label={\emph{(\alph*)}}]
        \item If 
        $ \frac{ n }{ \sigma^2 }
        \ll \frac{ 1 }{ \beta -\frac{ n }{ 2 } } $
        and $ N \geq 3 n \log n $,
        then
        \[
            \lim_{ n \to \infty } \Ex \mu(\tilde{P}_{N,n}^{\beta,\sigma}) = 1 .
        \]
        \item \label{forGauss} If 
        $ \frac{ 1 }{ \beta -\frac{ n }{ 2 } } 
        \ll \frac{ n }{ \sigma^2 }
        \ll \frac{ 1 }{ \sqrt{ \beta -\frac{ n }{ 2 } } } $,
        then,
        \[
            \lim_{ n \to \infty } \Ex \mu(\tilde{P}_{N,n}^{\beta,\sigma}) =  
            \begin{cases} 
            0 &\text{ if } N\ls \exp\left( ( 1 - \epsilon ) \frac{ n}{ \sigma^2 }( \beta - \frac{ n }{ 2 } ) \right) \\
            1 &\text{ if } N\gr \exp\left( ( 1 + \epsilon ) \frac{ n}{ \sigma^2 }( \beta - \frac{ n }{ 2 } ) \right).
            \end{cases}
        \]
        \item If %$ \frac{ \beta - \frac{ n }{ 2 } }{ \log \left( n \log ( \frac{ n }{ \sigma^2 } ) \right) } \to \infty $ and 
        $ \frac{ n }{ \sigma^2 } \to \infty $ and $ \sigma > e^{ - \frac{ n }{ 3 } } $ (in particular this holds for \(\sigma\equiv 1\)), then,
        \[
            \lim_{ n \to \infty } \Ex \mu(\tilde{P}_{N,n}^{\beta,\sigma}) = 
            \begin{cases} 
            0 &\text{ if } N\ls \exp\left( ( \beta - \frac{ n }{ 2 } ) \log \left( ( 1 - \epsilon ) \frac{ n }{ \sigma^2 } \right) \right) \\ 
            1 &\text{ if } N\gr \exp\left(  ( \beta - \frac{ n }{ 2 } ) \log \left( ( 1 + \epsilon ) \frac{ n }{ \sigma^2 } \right) \right) .
            \end{cases}
        \]
    \end{enumerate}
\end{theorem}
Since the densities of a sequence of beta-prime distributions with parameters $ \sigma^2 = 2 \beta \to \infty $ converge to the density of the standard multivariate Gaussian distribution, we also recover Pivovarov's threshold for Gaussian polytopes.
We state it here in a slightly more explicit form than in Theorem 2.2.1 from \cite{PivPhd}.
For a related result where the log concave isotropic measures are replaced by the volume ratios of the intersection of Gaussian polytopes with balls of arbitrary radii, see Theorem 2.1 from \cite{Piv}.
\begin{corollary}\label{cor.Gaussian}
Fix $\epsilon\in(0,1/2)$. Let \(\mu = \mu_n \) denote a sequence of isotropic log-concave measures on $\mathbb{R}^n$ and let $N=N(n)$ be a sequence of natural numbers.
Let \(X_1,\ldots,X_N\) be independent random points distributed according to the standard Gaussian distribution on \(\R^n\) and let \(G_{N,n}\coloneqq\mathrm{conv}\{X_1,\ldots,X_N\}\). 
Then,
        \[
            \lim_{ n \to \infty } \Ex \mu(G_{N,n}) = 
            \begin{cases} 
            0 &\text{ if } N\ls \exp\bigl(\bigl(\frac{1}{2} - \epsilon\bigr) n\bigr)\\
            1 &\text{ if } N\gr \exp\bigl(\bigl(\frac{1}{2} + \epsilon\bigr) n\bigr).
            \end{cases}
        \]
\end{corollary}

The proofs of the above statements can be found in Section \ref{sec.random.conv.hulls}. We stress that in all Theorems \ref{Theorem1}, \ref{thm.V_k.fixed_k} and \ref{thm.beta.prime.sigma}, the parameter $\beta$ is actually allowed to vary with the dimension $n$.

%It is noteworthy that the statements of Theorem \ref{Theorem1} and Theorem \ref{thm.beta.prime.sigma} interpolate in a certain sense between the results of \cite{Piv} and \cite{PivPhd} for the corresponding spherical and Gaussian models: The superexponential-order threshold given by Theorem \ref{Theorem1} fits with the threshold of the same order in \cite[Theorem 2.4]{Piv}. This is to be expected, since the uniform distribution on the sphere arises as the weak limit of the beta distribution, as $\beta\to -1$ (see the proof of Theorem 2.7 in \cite{GKT}). On the other hand, by letting $\beta\to \infty$ in either the beta or the beta-prime model, the Gaussian distribution on $\mathbb{R}^n$ can be recovered as a limiting case; see \cite[Remark 1.20]{KTT}. This connects Theorem \ref{thm.beta.prime.sigma} to \cite[Theorem 2.2.1]{PivPhd}.

Finally, in the spirit of \cite{Piv}, we also treat the dual setting, providing theorems of similar type for polytopes generated as intersections of random halfspaces. More precisely, given $X_1,\ldots,X_N$ chosen independently according to the beta or the beta-prime distribution, we consider the polytopes formed as intersections of the sets
\[
\{x\in \mathbb{R}^n : \langle X_i,x\rangle\ls a\}, \quad i=1,\ldots,N,
\]
for suitable $a>0$. The exact statement and proofs of these results can be found in the final Section \ref{sec.random.facets}.

\section{Notation and auxiliary estimates}\label{sec:notation}

We start this section by collecting some general notation. 
We work in the Euclidean space $\R^n$, $n\in \N = \{ 1 , 2 , \ldots \} $, equipped with the scalar product $\langle \cdot\,, \cdot \rangle$ and corresponding norm $\norm{\cdot}$. By $\B_2^n$ we denote the closed Euclidean unit ball and by $S^{n-1}$ the Euclidean unit sphere, equipped with the unique rotationally invariant probability measure $\sigma$. By $\de x$ we indicate Lebesgue integration in the appropriate dimension, and we use the symbol $\V{\cdot}$ for the volume, i.e., the $n$-dimensional Lebesgue measure. We abbreviate $\vball_n=\V{B_2^n}$. 
%Given $k\ls n$, the set of all $k$-dimensional subspaces of $\mathbb{R}^n$ is denoted by $G_{n,k}$, and comes equipped with the rotationally invariant probability measure $\nu_{n,k}$. 

By a convex body in $\mathbb{R}^n$ we mean a compact, convex subset of $\mathbb{R}^n$ with non-empty interior. 
%We refer the reader to the book \cite{SchneiderBuch} for the necessary background in the theory of convex bodies. 
Given a convex body $K$ in $\mathbb{R}^n$ and $t\gr 0$, the intrinsic volumes $(V_j(K))_{j=0}^n$, \(j\in\{0,\ldots,n\}\),  are the non-negative coefficients of the polynomial in \(t\) that appear in Steiner's formula, (see e.g.  \cite[Equation (4.2.27)]{SchneiderBuch}), namely
\[
V_n(K+tB_2^n) = \sum_{j=0}^n t^{n-j}\kappa_{n-j}V_j(K).
\]

Throughout the text, given two sequences of numbers positive real numbers \((a_n)_{n\in\mathbb{N}}\) and \((b_n)_{n\in\mathbb{N}}\) we will use the notation $ a_n \ll b_n $ for $ a_n = o ( b_n ) $, meaning that $ a_n/b_n \to 0 $, as $ n \to \infty $. Analogously, we will use \(a_n\gg b_n\) meaning $ a_n/b_n\to +\infty $, as $ n \to \infty $. Furthermore, we write $a_n \sim b_n$, if $a_n/b_n\to 1$, as $n\to \infty$.
Finally, we denote the set $\{1,\ldots,n\}$ by $[n]$.

\subsection{The beta and beta-prime distributions}\label{ssec.beta.def}

As aforementioned, our focus in this paper is on two specific classes of probability distributions on $\R^n$, namely, the beta and beta-prime distributions. To introduce the beta distribution, we set
	\[
	c_{n,\beta} := \pi^{-n/2}\frac{\Gamma\left(\beta+\frac{n}{2}+1\right)}{\Gamma(\beta+1)}, \quad\beta>-1,\, n\in \N,
	\]
	and define $\nu_\beta$ to be the probability measure on $B_2^n$ with density function
	\[
	p_{n,\beta}(x) := c_{n,\beta} (1-\norm{x}^2)^\beta,\quad x\in B_2^n.
	\]
	The corresponding one-dimensional marginal density function of $\nu_\beta$ is 
\[
f_{\beta}(t) := \alpha_{n,\beta} (1-t^2)^{\beta+\frac{n-1}{2}}, \qquad t\in [-1,1],
\]
where 
\begin{align*}
\alpha_{n,\beta}:= \frac{c_{n,\beta}}{c_{n-1,\beta}} = \pi^{-1/2}\frac{\Gamma\left(\beta+\frac{n}{2}+1\right)}{
		\Gamma\left(\beta+\frac{n+1}{2}\right)}.
\end{align*}
Finally, for $d\in [0,1]$, we abbreviate 
		\[
		\cdf(d) := \int_d^1 f_{\beta}(t)\de t.
		\]
To introduce the beta-prime distribution, we define
	\[
	    \tilde{c}_{n,\beta,\sigma}:=\sigma^{-n}\pi^{-n/2} \frac{\Gamma(\beta)}{\Gamma(\beta-\frac{n}{2})}, \qquad \beta>\frac{n}{2},\quad \sigma>0, \quad n \in \mathbb{N},
    \]
	and let $\tilde{\nu}_{\beta,\sigma}$ be the probability measure on $\R^n$ with density function
	\[
	    \tilde{p}_{n,\beta,\sigma}(x):=\tilde{c}_{n,\beta,\sigma}\left(1+\frac{\norm{x}^2}{\sigma^2}\right)^{-\beta}, \qquad x \in \R^n.
	\]
Moreover, let 
\begin{align*}
\tilde{\alpha}_{n,\beta,\sigma}:=\frac{\tilde{c}_{n,\beta,\sigma}}{\tilde{c}_{n-1,\beta,\sigma}}=\sigma^{-1} \pi^{-1/2} \frac{\Gamma(\beta-\frac{n-1}{2})}{\Gamma(\beta-\frac{n}{2})}, 
\end{align*}
so that 
	\[
	\tilde{f}_{\beta,\sigma}(t):=\tilde{\alpha}_{n,\beta,\sigma}(1+t^2)^{-\beta+\frac{n-1}{2}},\qquad t\in\R,
	\]
	is the one-dimensional marginal density function of $\tilde{\nu}_{\beta,\sigma}$. Analogously to the beta case, for $d \in [0,\infty)$, we denote 
	\[
	\tilde{\cdf}(d):=\int_d^\infty \tilde{f}_{\beta,\sigma}(t) \de t.
	\]
Estimates on the asymptotic behavior of the distribution functions of $\nu_\beta$ and $\tilde{\nu}_{\beta,\sigma}$, in particular for the functions $\cdf$ and $\tilde{\cdf}$ defined above, play a central role in our work.	We begin with a bound on the ratio of Gamma functions, that is a particular case of Wendel's inequality (see e.g. eq. \((7)\) in \cite{Wendel}), but written in a similar form already in \cite{Artin}.

\begin{lemma}
    \label{lem:Gammabounds}
    For every \(x>1\),
    \[
    \rec{\sqrt{x}}<\frac{\Gamma(x)}{\Gamma(x+\rec{2})}<\rec{\sqrt{x-1}}.
    \]
\end{lemma}
The previous inequalities are used in the proof of the following bounds for the distribution function \(\cdf\).

\begin{lemma}\label{lem:B.bounds}
Let $d\in (0,1)$. Then,
\begin{equation*}
\frac{1}{2\sqrt{\pi}}\frac{(1-d^2)^{\beta+\frac{n+1}{2}}}{\sqrt{\beta+\frac{n}{2}+1}} < \cdf(d) < \frac{1}{2d\sqrt{\pi}}\frac{(1-d^2)^{\beta+\frac{n+1}{2}}}{\sqrt{\beta+\frac{n}{2}}}.
\end{equation*}
\end{lemma}

\begin{proof}
Using the change of variable $s=1-t^2$, we write
\[
\cdf(d) = \alpha_{n,\beta} \int_d^1 (1-t^2)^{\beta + \frac{n-1}{2}} \de t = \frac{1}{2}\alpha_{n,\beta} \int_0^{1-d^2} s^{\beta+\frac{n-1}{2}}(1-s)^{-\frac{1}{2}}\de s.
\]
Note that since $s\in(0,1-d^2)$, we have $(1-s)^{-1/2} \in (1,d^{-1})$, so
\[
\frac{\alpha_{n,\beta}}{2}\int_0^{1-d^2}s^{\beta+\frac{n-1}{2}}\de s < \cdf(d) < \frac{\alpha_{n,\beta}}{2d}\int_0^{1-d^2}s^{\beta+\frac{n-1}{2}}\de s.
\]
The fact that 
\begin{align*}
\frac{\alpha_{n,\beta}}{\beta+\frac{n+1}{2}}=\rec{\sqrt{\pi}}\frac{\Gamma(\beta+\frac{n}{2}+1)}{\Gamma(\beta+\frac{n}{2}+\frac{3}{2})},
\end{align*}
together with \Cref{lem:Gammabounds}, completes the proof.
\end{proof}

\begin{remark}
Note that an adaptation of the above proof yields similar estimates on the growth of $\tilde{\cdf}$ if the parameter $\sigma$ is an absolute constant. For instance if $\sigma=1$, one has that
\begin{equation}\label{eq.tilde(F).sigma=1.bounds}
\frac{1}{2\sqrt{\pi}} \frac{(1+d^2)^{-\beta+\frac{n}{2}}}{\sqrt{\beta-\frac{n-1}{2}}} < \tilde{\cdf}(d) < \frac{1}{\sqrt{2\pi}} \frac{(1+d^2)^{-\beta+\frac{n}{2}}}{\sqrt{\beta-\frac{n+1}{2}}}
\end{equation}
for every $d>1$. Yet, in the general case where the parameter $\sigma$ could vary with $\beta$ or $n$ we will show that the asymptotic behaviour of $\tilde{\cdf}$ in terms of $\sigma, \beta$ and $n$ actually depends on the growth rate of the quantity $n/\sigma^2$. This will result to the different threshold results in the statement of Theorem \ref{thm.beta.prime.sigma}.
\end{remark}

To deal with the distribution function $\tilde{\cdf}$ for an arbitrary $\sigma>0$, we will use a different argument. Note first that a suitable substitution provides
\begin{equation} \label{eq.cdf.rewritten.aux}
    \tilde{ \cdf } ( d ) 
    = \frac{ \tilde{\alpha}_{n,\beta,\sigma} }{ \sqrt{ 2 b_n } } \int_{ a_n }^\infty \left( 1 + \frac{ s^2 }{ 2 b_n } \right)^{ - b_n } \de{ s },
    \quad \text{$ b_n = \beta - \frac{ n - 1 }{ 2 } $ and $ a_n =  d \frac{ \sqrt{ 2 b_n } }{ \sigma } $} .
\end{equation}
It is easy to see that $ \frac{ \tilde{\alpha}_{n,\beta,\sigma} }{ \sqrt{ 2 b_n } } \to \frac{1}{\sqrt{2\pi}} $ whenever $ b_n \to \infty $.
The estimates of $ \tilde{ \cdf } ( d ) $ which will appear in the proof of Theorem \ref{thm.beta.prime.sigma} are based on \eqref{eq.cdf.rewritten.aux} and the following lemma.
\begin{lemma} \label{lem.tail.Ftilda}
    Let $ ( a_n )_{ n \in \N } $, $ ( b_n )_{ n \in \N } $ be two sequences with $ a_n \geq 0 $ and $ \frac{ 1 }{ 2 }<b_n\to\infty$.
    \begin{enumerate}[label={\emph{(\alph*)}}]
        \item If $ \frac{ a_n^4 }{ b_n } \to 0 $, then, 
        \[ \int_{ a_n }^\infty \left( 1 + \frac{ t^2 }{ 2 b_n } \right)^{ - b_n } \de{ t } 
        \sim \int_{ a_n }^\infty e^{ - \frac{ t^2 }{ 2 } } \de{ t } .\]
        If additionally $ a_n \to \infty $, then, 
        \[ \int_{ a_n }^\infty \left( 1 + \frac{ t^2 }{ 2 b_n } \right)^{ - b_n } \de{ t } 
        \sim \frac{ e^{ - \frac{ a_n^2 }{ 2 } } }{ a_n } .\]
        \item If $ \frac{ a_n^2 }{ b_n } \to \infty $, then, 
        \[ \int_{ a_n }^\infty \left( 1 + \frac{ t^2 }{ 2 b_n } \right)^{ - b_n } \de{ t } 
        \sim \frac{ 1}{ \sqrt{ 2b_n }} \left( 1 + \frac{ a_n^2 }{ 2 b_n } \right)^{ - ( b_n - \frac{ 1 }{ 2 } ) } . \]
    \end{enumerate}
\end{lemma}

The main ingredient to prove Lemma \ref{lem.tail.Ftilda} is Laplace's method.
We refer the reader to Theorem 1.1 of \cite{Wong} for a more general version of the following lemma.
\begin{lemma}[Special case of Laplace's method] \label{lem:Laplace}
    Let $ h \colon [ a , \infty ) \to \R $ be a strictly increasing and differentiable function.
    Then, as $ \lambda \to \infty $,
    \[
        \int_a^\infty e^{ - \lambda h ( t ) } \de{ t }
        \sim \frac{ e^{ - \lambda h ( a ) } }{ \lambda h' ( a ) } .
    \]
\end{lemma}

% \begin{proof}
%   This is a well know result.
%   See e.g.\ Theorem 1.1. page 58 of \cite{Wong}.
% %    We apply Theorem 1.1. page 58 of \cite{Wong}.
% %    For the setting of this Theorem, we have $ \mu = 1 $, $ a_0 = h' ( a ) $, $ b_0 = 1 $ and $ \alpha = 1 $.
% %    This gives $ c_0 := \frac{ b_0 }{ \mu a_0^{ \alpha / \mu } } = 1 / h' ( a ) $. 
% %    In particular the first term of the series $ \sum_{ s = 0 }^\infty \Gamma \left( \frac{ s + \alpha }{ \mu } \right) \frac{ c_s }{ \lambda^{ ( s + \alpha ) / \mu } } $ is equal to $ 1 / ( \lambda h' ( a ) ) $
% \end{proof}

\begin{proof}[Proof of Lemma \ref{lem.tail.Ftilda}]
    We first show a pair of auxiliary estimates.
    The inequality
    $ x - \frac{ x^2 }{ 2 }  
    \leq \log \left( 1 + x \right)
    \leq x $
    gives that
    \[ 1 
    \leq \frac{ \left( 1 + \frac{ t^2 }{ 2 b_n } \right)^{ - b_n } }{ e^{ - \frac{ t^2 }{ 2 } } }
    = \exp \left( - b_n \log \left( 1 + \frac{ t^2 }{ 2 b_n } \right) + \frac{ t^2 }{ 2 } \right)
    \leq e^{ \frac{ t^4 }{ 8 b_n } } . \]
    Therefore, for any couple of sequences $ 0 \leq c_n < d_n $, we have
    \[
        \int_{ c_n }^{ d_n } e^{ - \frac{ t^2 }{ 2 } } \de{ t }  
        \leq \int_{ c_n }^{ d_n } \left( 1 + \frac{ t^2 }{ 2 b_n } \right)^{ - b_n } \de{ t } 
        \leq e^{ \frac{ d_n^4 }{ 8 b_n } } \int_{ c_n }^{ d_n } e^{ - \frac{ t^2 }{ 2 } } \de{ t } ,
    \]
    and in particular
    \begin{equation} \label{eq.approx.small.cn}
        \frac{ d_n^4 }{ b_n } \to 0
        \ \Rightarrow \ 
        \int_{ c_n }^{ d_n } \left( 1 + \frac{ t^2 }{ 2 b_n } \right)^{ - b_n } \de{ t } 
        \sim \int_{ c_n }^{ d_n } e^{ - \frac{ t^2 }{ 2 } } \de{ t } .
    \end{equation}
    If additionally $ c_n \to \infty $ we can get a more explicit approximation by using a substitution and the Laplace's method.
    The new estimate is
    \begin{equation} \label{eq.approx.small.cn.infty}
        \frac{ d_n^4 }{ b_n } \to 0 \text{ and } c_n \to \infty 
        \ \Rightarrow \ 
        \int_{ c_n }^{ d_n } \left( 1 + \frac{ t^2 }{ 2 b_n } \right)^{ - b_n } \de{ t } 
        \sim \frac{ e^{ - \frac{ c_n^2 }{ 2 } } }{ c_n } - \frac{ e^{ - \frac{ d_n^2 }{ 2 } } }{ d_n } .
    \end{equation}
    
    Since for any $ t $, the map $ ( \frac{ 1 }{ 2 } , \infty ) \ni b \mapsto \left( 1 + \frac{ t^2 }{ 2 b } \right)^{ - b } $ is decreasing, we have that for any sequence $ ( c_n )_{ n \in \N } $ with $ \frac{ 1 }{ 2 } < c_n^2 < b_n $,
    \begin{equation*}
        \begin{split}
            \int_{ c_n }^\infty \left( 1 + \frac{ t^2 }{ 2 b_n } \right)^{ - b_n } \de{ t }
            & \leq \int_{ c_n }^\infty \left( 1 + \frac{ t^2 }{ 2 c_n^2 } \right)^{ - c_n^2 } \de{ t }
            \\ & = \sqrt{ 2 } c_n \int_{ \frac{ 1 }{ \sqrt{ 2 } } }^\infty \left( 1 + s^2 \right)^{ - c_n^2 } \de{ s }
            \\ & = \sqrt{ 2 } c_n \int_{ \frac{ 1 }{ \sqrt{ 2 } } }^\infty \exp \left(  - c_n^2 \log ( 1 + s^2 ) \right) \de{ s } .
        \end{split}
    \end{equation*}
    Laplace's method, see Lemma \ref{lem:Laplace}, now implies that for $ c_n \to \infty $,
    \begin{equation*}
        \int_{ \frac{ 1 }{ \sqrt{ 2 } } }^\infty \exp \left(  - c_n^2 \log ( 1 + s^2 ) \right) \de{ s }
        \sim \frac{ \exp \left( - c_n^2 \log ( \frac{ 5 }{ 4 } ) \right) }{ c_n^2  } 
    \end{equation*}
    In particular,
    \begin{equation} \label{eq.approx.small.cn.tail}
         b_n > c_n^2 \text{ and } c_n \to \infty
         \Rightarrow \int_{ c_n }^\infty \left( 1 + \frac{ t^2 }{ 2 b_n } \right)^{ - b_n } \de{ t }
         = o \left( e^{ - c_n^2 \log ( \frac{ 5 }{ 4 } ) } \right) . 
    \end{equation}
    Now we have all the ingredients to show Lemma \ref{lem.tail.Ftilda} (a), but we need to distinguish the case where $ a_n \to \infty $ from the case where $ a_n $ is bounded.
    
    First, we assume that $ a_n $ is bounded.
    Let $ c_n > a_n $ be a sequence such that $ \frac{ c_n^4 }{ b_n } \to 0 $ and $ c_n \to \infty $. 
    Splitting the integral in two parts and applying \eqref{eq.approx.small.cn} with $ a_n $ and $ c_n $, gives
    \begin{equation*}
        \begin{split}
            \int_{ a_n }^\infty \left( 1 + \frac{ t^2 }{ 2 b_n } \right)^{ - b_n } \de{ t }
            & = \int_{ a_n }^{ c_n } \left( 1 + \frac{ t^2 }{ 2 b_n } \right)^{ - b_n } \de{ t } 
            + \int_{ c_n }^{ \infty } \left( 1 + \frac{ t^2 }{ 2 b_n } \right)^{ - b_n } \de{ t }
            \\ & \sim \int_{ a_n }^{ c_n } e^{ - \frac{ t^2 }{ 2 } } \de{ t }  +  o ( 1 )
            \\ & \sim \int_{ a_n }^{ \infty } e^{ - \frac{ t^2 }{ 2 } } \de{ t } .
        \end{split}
    \end{equation*}
    
    Second, we assume that $ a_n \to \infty $.
    We split the integral in two parts and use the estimates \eqref{eq.approx.small.cn.infty} with $ c_n = a_n $ and $ d_n = 2 a_n $, and \eqref{eq.approx.small.cn.tail} with $ c_n = 2 a_n $.
    This gives
    \begin{equation*}
        \begin{split}
            \int_{ a_n }^\infty \left( 1 + \frac{ t^2 }{ 2 b_n } \right)^{ - b_n } \de{ t }
            & = \int_{ a_n }^{ 2 a_n } \left( 1 + \frac{ t^2 }{ 2 b_n } \right)^{ - b_n } \de{ t } 
            + \int_{ 2 a_n }^{ \infty } \left( 1 + \frac{ t^2 }{ 2 b_n } \right)^{ - b_n } \de{ t }
            \\ & \sim \frac{ e^{ - \frac{ a_n^2 }{ 2 } } }{ a_n } - \frac{ e^{ - 2 a_n^2 } }{ 2 a_n } +  o \left( e^{ - 4 a_n^2 \log ( \frac{ 5 }{ 4 } ) } \right)
            \\ & \sim \frac{ e^{ - \frac{ a_n^2 }{ 2 } } }{ a_n } .
        \end{split}
    \end{equation*}
    
    \bigskip
    To show part (b) of Lemma \ref{lem.tail.Ftilda}, note that the substitution $ s = ( 1 + \frac{ t^2 }{ 2 b_n } )^{ - 1 } $ gives
    \begin{equation*}
        \begin{split}
            \int_{ a_n }^\infty \left( 1 + \frac{ t^2 }{ 2 b_n } \right)^{ - b_n } \de{ t } 
            & = \sqrt{ \frac{ b_n }{ 2 } } \int_0^{ \left( 1 + \frac{ a_n^2 }{ 2 b_n } \right)^{ - 1 } } s^{ b_n - \frac{ 3 }{ 2 } } ( 1 - s )^{ - \frac{ 1 }{ 2 } } \de{ s } .
        \end{split}
    \end{equation*}
    But, since $ ( 0 , 1 ) \ni s \mapsto ( 1 - s )^{ -\frac{ 1 }{ 2 } } $ is increasing, we have
    \begin{equation*}
        \begin{split}
            1
            \leq  \frac{ \sqrt{ \frac{ b_n }{ 2 } } \int_0^{ \left( 1 + \frac{ a_n^2 }{ 2 b_n } \right)^{ - 1 } } s^{ b_n - \frac{ 3 }{ 2 } } ( 1 - s )^{ - \frac{ 1 }{ 2 } } \de{ s } }{ \sqrt{ \frac{ b_n }{ 2 } } \int_0^{ \left( 1 + \frac{ a_n^2 }{ 2 b_n } \right)^{ - 1 } } s^{ b_n - \frac{ 3 }{ 2 } } \de{ s } }
            \leq \left( 1 - \left( 1 + \frac{ a_n^2 }{ 2 b_n } \right)^{ - 1 } \right)^{ - \frac{ 1 }{ 2 } }  .
        \end{split}
    \end{equation*}
    Observe that the right hand side of the last equation tends to $ 1 $ because $ \frac{ a_n^2 }{ b_n } \to \infty $.
    Thus, the two last equations provide the equivalence
    \begin{equation*}
        \begin{split}
            \int_{ a_n }^\infty \left( 1 + \frac{ t^2 }{ 2 b_n } \right)^{ - b_n } \de{ t } 
            & \sim \sqrt{ \frac{ b_n }{ 2 } } \int_0^{ \left( 1 + \frac{ a_n^2 }{ 2 b_n } \right)^{ - 1 } } s^{ b_n - \frac{ 3 }{ 2 } } \de{ s } 
            \\ & \sim \frac{ \sqrt{ b_n } }{ \sqrt{ 2 } ( b_n - \frac{ 1 }{ 2 } ) } \left( 1 + \frac{ a_n^2 }{ 2 b_n } \right)^{ - ( b_n - \frac{ 1 }{ 2 } ) }\\
       &\sim \frac{ 1}{ \sqrt{ 2b_n }}  \left( 1 + \frac{ a_n^2 }{ 2 b_n } \right)^{ - ( b_n - \frac{ 1 }{ 2 } ) } , \end{split}
    \end{equation*}
    which completes the proof.
\end{proof}

\subsection{Isotropic log-concave probability measures} \label{sec:iso.log.conv.prob.meas}

A probability measure $\mu$ on $\R^n$ is called \emph{log-concave}, if for all compact subsets $A,B$ of $\R^n$ and all $\lambda\in (0,1)$,
\[
\mu((1-\lambda)A + \lambda B) \gr \mu(A)^{1-\lambda}\mu(B)^{\lambda}.
\]
It is called \emph{isotropic}, if its center of mass is at the origin, i.e., 
\[
\int_{\R^n} \langle x,\theta\rangle\de \mu(x) = 0
\]
holds for every $\theta\in S^{n-1}$, and satisfies the isotropic condition, that is,
\[
\int_{\R^n} \langle x,\theta\rangle^2 \de \mu(x) = 1
\]
for all $\theta\in S^{n-1}$.

In the sequel, we will rely on the so-called thin-shell concentration property of isotropic log-concave probability measures. Answering a central question in asymptotic convex geometry (see \cite{ABP}), Klartag \cite[Theorem 1.4]{Kla} proved that an isotropic log-concave measure is typically concentrated on a ``thin spherical shell'' around the Euclidean ball of radius $\sqrt{n}$. The statement reads as follows. 
\begin{theorem}[Thin shell concentration]\label{thm:concentration}
Let $\mu$ be an isotropic log-concave probability measure in $\mathbb{R}^n$. Then, for every $\epsilon\in(0,1)$,
\begin{equation}\label{eq.klartag.thinshell}
\mu\bigl(\bigl\{x\in\R^n : \abs[\big]{\norm{x}-\sqrt{n}} \gr\epsilon\sqrt{n}\}) \ls Cn^{-c\epsilon^2},
\end{equation}
for some absolute constants $c, C>0$.
\end{theorem}
Results of this type are closely linked to the long-standing thin shell conjecture, which asks, if the quantity $\Ex \bigl(\norm{X}-\sqrt{n}\bigr)^2$ can be uniformly bounded by a constant independent of the dimension, for any random vector $X$ distributed according to an isotropic and log-concave probability measure on $\R^n$. For more information on the history of this problem, recent improvements of \Cref{thm:concentration}, as well as the general theory of isotropic log-concave probability measures, one can consult the monograph \cite{BGVV}. 

	\section{Convex hulls of random points}\label{sec.random.conv.hulls}

Recall that by $P_{N,n}^\beta$ and $\tilde{P}_{N,n}^{\beta,\sigma}$ we denote the convex hulls arising from $N>n$ independent random points in $\R^n$, distributed according to the beta distribution with parameter \(\beta\) and the beta-prime distribution with parameters \(\beta\) and \(\sigma\), respectively.

\subsection{Preparatory lemmas}\label{subsec.lemmas}

The proofs of Theorem \ref{Theorem1} and Theorem \ref{thm.beta.prime.sigma} follow the method introduced in \cite{DFMcD} and exploited in \cite{Piv}. 
We thus define, for every $x\in \R^n$, the functions
\[
q(x) := \inf\{\Pr{X\in H} : H \text{ is a halfspace containing } x\},
\]
when \(X\sim\nu_\beta\), and
\[
\tilde{q}(x):=\inf \{ \Pr{X \in H}: H \text{ is a halfspace containing } x\},
\]
when \(X\sim\tilde\nu_{\beta,\sigma}\). The following lemma implies a way to compute $q(x)$ and $\tilde{q}(x)$ in terms of the Euclidean norm of the point $x\in\R^n$.

\begin{lemma}\label{lem.q=B}
Let $H$ be a halfspace at distance $d\ge 0$ from the origin. Then,
\begin{enumerate}[label={\emph{(\alph*)}}]
\item$\Pr{X\in H} = \cdf(d)$, when \(X\sim\nu_\beta\),
\item $\Pr{X\in H} = \tilde{\cdf}(d)$, when \(X\sim\tilde\nu_{\beta,\sigma}\).
\end{enumerate}
\end{lemma}

\begin{proof}
		We prove the lemma only for the case (a), since (b) is analogous. By rotational invariance of the measure $\nu_\beta$, we may assume that $H = \{x=(x_1,\ldots,x_n) \in \R^n : x_1\geq d\}$. We write
		\begin{equation*}
			\begin{split}
				\Pr{X\in H} &= \nu_\beta(H) = \int_H p_{n,\beta}(x)\de x = c_{n,\beta}\int_H (1-\norm{x}^2)^\beta\de x\\
				&= c_{n,\beta} \int_d^1\,\int_{B_2^{n-1}} (1-\norm{x}^2)^\beta\de (x_2,\ldots,x_n)\de x_1\\
				&= c_{n,\beta} \int_d^1\,\int_{B_2^{n-1}} (1-t^2)^\beta\left(1-\frac{\norm{y}_2^2}{1-t^2}\right)^\beta\de y\de t\\
				&= c_{n,\beta} \int_d^1 (1-t^2)^\beta \int_{B_2^{n-1}} (1-\norm{z}_2^2)^\beta(1-t^2)^\frac{n-1}{2}\de z\de t\\
				&= \alpha_{n,\beta} \int_d^1 (1-t^2)^{\beta + \frac{n-1}{2}} \int_{B_2^{n-1}} p_{n-1,\beta}(z)\de z\de t\\
				&= \int_d^1 f_{\beta}(t)\de t =\cdf(d),
			\end{split}
		\end{equation*}
		which concludes the proof.
\end{proof}

\begin{corollary}\label{corol.q=B}
For every $x\in \R^n$,
\begin{enumerate}[label={\emph{(\alph*)}}]
\item $q(x) = \cdf(\norm{x})$,
 \item $\tilde{q}(x) = \tilde{\cdf}(\norm{x})$.
 \end{enumerate}
\end{corollary}

\begin{proof}
As before, we discuss only the case (a). Note that $q(0) = 1/2 =\cdf(0)$. If $x\neq 0$, let $H(x)$ be the halfspace bounded by the tangent hyperplane to $\norm{x}B_2^n$ at $x$, not containing $0$. Then, by Lemma \ref{lem.q=B} (a), we have
		\[
		\cdf(\norm{x}) = \Pr{X\in H(x)} \geq q(x).
		\]
		Conversely, let $H$ be a halfspace at distance $d$ from the origin, such that $x\in H$. If $d=0$, then, $\Pr{X\in H} \geq 1/2 \geq \mathrm{F}(\norm{x})$. If $d>0$, then, again by Lemma \ref{lem.q=B} (a), we have $\Pr{X\in H} =\cdf(d) \geq \cdf(\norm{x})$, since $d\leq \norm{x}$. It follows that $q(x)\geq \cdf(\norm{x})$.
\end{proof}

Using Corollary \ref{corol.q=B}, we can relate the probability content of the random polytopes $P_{N,n}^\beta$ and $\tilde{P}_{N,n}^{\beta,\sigma}$ to the distribution functions $\cdf$ and $\tilde{\cdf}$, respectively. In particular, we upper bound the expected volume of $P_{N,n}^\beta$ in terms of $\cdf$ (and similarly for $\tilde{P}_{N,n}^{\beta,\sigma}$).

\begin{lemma}\label{lemma.ineq}
		Let $A$ be a bounded, measurable subset of $\R^n$.
		\begin{enumerate}[label={\emph{(\alph*)}}]
		\item In the beta model,
		\[
		\Pr{A\subseteq P_{N,n}^\beta}\ls\frac{\Ex \V{P_{N,n}^\beta\cap A}}{\V{A}}\leq N\sup_{x\in A}\cdf(\norm{x}).
		\]
		\item In the beta-prime model,
		\[  
            \mu(A)\Pr{A\subseteq \tilde{P}_{N,n}^{\beta,\sigma}}\ls\Ex  \mu(\tilde{P}_{N,n}^{\beta,\sigma} \cap A) \leq N \mu(A) \sup_{x\in A}\tilde{\cdf}(\norm{x}),
        \]
        where $\mu$ is any isotropic log-concave probability measure on $\R^n$.
		\end{enumerate}
	\end{lemma}
	
	\begin{proof}
	(a) First, note that, for any $x\in P_{N,n}^\beta=\mathrm{conv}\{X_1,\ldots,X_N\}$ and any halfspace $H$ containing $x$, there must be some $X_i \in H$. This implies that
	\[
	\{x\in P_{N,n}^\beta\} \subseteq \bigcup_{i=1}^N \{X_i \in H\}.
	\]
	Since the previous inclusion holds for any halfspace $H$ containing $x$, by a union bound and \Cref{corol.q=B} (a), we get that 
	\begin{align*}
	\mathbb{P}(x\in P_{N,n}^\beta)\ls N q(x) = N \cdf(\norm{x}). 
	\end{align*}
	Now, using the latter estimate,
	\[
	\Ex \V{P_{N,n}^\beta\cap A} = \Ex \int_A \indic{P_{N,n}^\beta}(x)\de x = \int_A \Pr{x\in P_{N,n}^\beta}\de x \le N\V{A}\sup_{x\in A}\cdf(\norm{x}).
	\] 
This proves the upper bound. On the other hand, since the event $ \{ A \subseteq P_{N,n}^\beta \} $ implies $ \{ \V{A} \le \V{P_{N,n}^\beta\cap A} \} $, Markov's inequality gives
	\begin{align*}
	\V{A}\Pr{A\subseteq P_{N,n}^\beta} \le \Ex \V{P_{N,n}^\beta\cap A},
	\end{align*}
	completing the proof.
  
  (b) The proof follows along the same line as (a), using now Corollary \ref{corol.q=B} (b) instead of (a).
  As above, for any halfspace $H$ and any point $x\in H$, we have
 	\[
 	\{x\in \tilde{P}_{N,n}^{\beta,\sigma}\} \subseteq \bigcup_{i=1}^N \{X_i \in H\}.
 	\]
  Again, by a union bound and \Cref{corol.q=B} (b), we get that 
 	\begin{align*}
 	\mathbb{P}(x\in \tilde{P}_{N,n}^{\beta,\sigma})\ls N \tilde{q}(x) = N \tilde{\cdf}(\norm{x}). 
 	\end{align*}
 	Using this estimate, we have
 	\[
 	\Ex \mu(\tilde{P}_{N,n}^{\beta,\sigma}\cap A) = \Ex \int_A \indic{\tilde{P}_{N,n}^{\beta,\sigma}}(x) \,\mu(\de x) = \int_A \Pr{x\in \tilde{P}_{N,n}^{\beta,\sigma}} \, \mu(\de x) \le N\mu(A)\sup_{x\in A}\tilde{\cdf}(\norm{x}),
 	\] 
  which proves the upper bound. Finally, by Markov's inequality, we get the lower bound
 	\begin{align*}
 	\mu(A)\Pr{A\subseteq \tilde{P}_{N,n}^{\beta,\sigma}} \le \Ex \mu(\tilde{P}_{N,n}^{\beta,\sigma}\cap A),
 	\end{align*}
 	finishing the proof.
	\end{proof}

Next, we reproduce an analogous ``ball inclusion'' argument as in \cite{DFMcD} in our setting. 

\begin{lemma}\label{lem.prob}
		\begin{enumerate}[label={\emph{(\alph*)}}]
		\item For any $R\in(0,1)$, the inclusion $RB_2^n\subseteq P_{N,n}^\beta$ holds with probability greater than $1-2\binom{N}{n}(1-\cdf(R))^{N-n}$. 
		\item For any $R>0$, the inclusion $RB_2^n\subseteq \tilde{P}_{N,n}^{\beta,\sigma}$ holds with probability greater than $1-2\binom{N}{n}(1-\tilde{\cdf}(R))^{N-n}$. 
		\end{enumerate}
	\end{lemma}
	
	\begin{proof}
	Let us start with part (a). Let $J\subseteq \{1,\ldots,N\}$ with $|J|=n$. With probability equal to one, the set $\{X_j\}_{j\in J}$ is affinely independent. Let $H_J$ be the affine hyperplane defined by the affine hull of $\{X_j\}_{j\in J}$ and $H_J^+, H_J^-$ be the corresponding closed halfspaces, determined by $H_J$. Moreover, let $X$ be an additional independent beta-distributed random point and let $E_J$ be the event, that, either $P_{N,n}^\beta\subseteq H_J^+$ and $\mathbb{P}(X\notin H) |_{H=H_J^+} \gr\cdf(R)$, or $P_{N,n}^\beta\subseteq H_J^-$ and $\mathbb{P}(X\notin H) |_{H=H_J^-}\gr\cdf(R)$. Note that here, and in the following, $\mathbb{P}(X\notin H) |_{H=G}$ denotes the evaluation of the map $H \mapsto \mathbb{P}(X\notin H)$ for the halfspace $G\subset \mathbb{R}^n$.
	
	Suppose that $RB_2^n\nsubseteq P_{N,n}^\beta$, so there exists some $x_0\in RB_2^n\setminus P_{N,n}^\beta$. Then, there exists some $J\subseteq \{1,\ldots,N\}$ with $\abs{J}=n$ such that either $P_{N,n}^\beta\subseteq H_J^+$ and $x_0\in H_J^-$ or $P_{N,n}^\beta\subseteq H_J^-$ and $x_0\in H_J^+$. Note that we have $\mathbb{P}(X\notin H)|_{H=H_J^+}\gr q(x_0)\gr\cdf(R)$, or $\mathbb{P}(X\notin H)|_{H=H_J^-}\gr q(x_0)\gr\cdf(R)$ respectively, since $\norm{x_0}\ls R$. It follows that
	\[
	\{RB_2^n\nsubseteq P_{N,n}^\beta\} \subseteq \bigcup_{\substack{J\subseteq [N] \\ \abs{J}=n}} E_J.
	\]
Clearly, using the union bound,
\[
\mathbb{P}(RB_2^n\nsubseteq P_{N,n}^\beta) \ls \binom{N}{n}\mathbb{P}(E_{[n]}).
\]
Next, note that $\mathbb{P}(X\notin H)|_{H=H_{[n]}^+}\gr \cdf(R)$ implies $\mathbb{P}(X\in H)|_{H=H_{[n]}^+} \ls 1-\cdf(R)$, and similarly for $H_{[n]}^-$. It follows that $\Pr{E_{[n]}\mid X_1,\ldots,X_n}\ls 2(1-\cdf(R))^{N-n}$. Finally, we get that $\Pr{E_{[n]}}=\Exval{\Pr{E_{[n]}\mid X_1,\ldots, X_n}}\ls 2(1-\cdf(R))^{N-n}$, and, hence,
\[
\mathbb{P}(RB_2^n\nsubseteq P_{N,n}^\beta) \ls 2\binom{N}{n}(1-\cdf(R))^{N-n},
\]
proving the statement of the lemma. The proof of part (b) is a word-by-word repetition of the proof of (a), where now $\tilde{\cdf}$ plays the role of $\cdf$.
\end{proof}

Finally, we provide an essential lemma for the proofs of Theorem \ref{Theorem1} and Theorem \ref{thm.beta.prime.sigma}.

\begin{lemma} \label{lemma.beta.prime.productNF} Let $ \epsilon > 0 $ be fixed.
  \begin{enumerate}[label={\emph{(\alph*)}}]
		\item In the beta model,
		  \begin{equation*}
        \lim_{n\to\infty} \frac{\Ex \V{P_{N,n}^\beta}}{\vball_n} = 
        \begin{cases} 
        0 & \text{ if } N \cdf \big( \sqrt{ 1 - n^{ - ( 1 - \epsilon ) } } \big) \to 0 \\
        1 & \text{ if } N \cdf \big( \sqrt{ 1 - n^{ - ( 1 + \epsilon ) } } \big) - n \log N \to \infty .
        \end{cases}
      \end{equation*}
		\item In the beta-prime model,
      \begin{equation*}
        \lim_{n\to\infty} \Ex \mu_n(\tilde{P}_{N,n}^{\beta,\sigma}) =
        \begin{cases} 
        0 &\text{ if } N \tilde{ \cdf } \big( ( 1 - \epsilon ) \sqrt{ n } \big) \to 0 \\
        1 &\text{ if } N \tilde{ \cdf } \big( ( 1 + \epsilon ) \sqrt{ n } \big) - n \log N \to \infty .
        \end{cases}
      \end{equation*}
  \end{enumerate}
\end{lemma}
\begin{proof}
    (a) Set $r_n=\sqrt{1-n^{-(1-\epsilon)}}$, $A_n=B_2^n\setminus r_nB_2^n$ and assume that $ N \cdf ( r_n ) \to 0 $.
    We have that $\sup_{x\in A_n}\cdf(\norm{x})\ls \cdf(r_n)$, since $\norm{x}\gr r_n$ for every $x\in A_n$. By using this, in conjunction with Lemma \ref{lemma.ineq} (a),
    \[
    \frac{\Ex \V{P_{N,n}^\beta\cap A_n}}{\vball_n} 
    \ls \frac{\Ex \V{P_{N,n}^\beta\cap A_n}}{\V{A_n}} \ls N\cdf(r_n) 
    \to 0  .
    \]
    Note also that
    \(
    \frac{\V{r_n B_2^n}}{\V{B_2^n}} = r_n^n \to 0.
    \)
    Thus
    \[
      \frac{\Ex \V{P_{N,n}^\beta}}{\vball_n}
      \leq \frac{\V{r_n B_2^n}}{\V{B_2^n}} + \frac{\Ex \V{P_{N,n}^\beta\cap A_n}}{\vball_n} 
      \to 0 .
    \]
    
    \medskip
    Now, we set $s_n =  \sqrt{ 1 - n^{ - ( 1 + \epsilon ) } } $ and assume that \( N \cdf ( s_n ) - n \log N \to \infty \).
    From the lower bound in \Cref{lemma.ineq} (a) with \(A= s_n B_2^n\) we get that
    \[
   	  \frac{\Ex \V{P_{N,n}^\beta}}{\vball_n} 
      \gr s_n^n \mathbb{P}(s_n B_2^n\subseteq P_{N,n}^\beta)
      \sim \mathbb{P}(s_n B_2^n\subseteq P_{N,n}^\beta) .
   	\]
    Hence, it suffices to show that
    \begin{align}\label{wert2}
    \lim_{n\to \infty}\Pr{s_nB_2^n \nsubseteq P_{N,n}^\beta} = 0.
    \end{align}
    By Lemma \ref{lem.prob} (a), we have, using $\binom{N}{n}\ls (eN/n)^n$,
    \begin{align*}
    \Pr{s_nB_2^n \nsubseteq P_{N,n}^\beta} &\ls 2 \binom{N}{n}(1-\cdf(s_n))^{N-n} \\
                                    &\ls 2(eN/n)^n\exp((N-n)\log(1-\cdf(s_n))) \\
                                    &= 2\exp\bigl(n\log(eN/n)+(N-n)\log(1-\cdf(s_n))\bigr).
    \end{align*}	
    Since $ \log ( 1 - x ) \leq - x $, we have
     \begin{align*}
        \Pr{s_nB_2^n \nsubseteq P_{N,n}^\beta} 
        &\leq 2\exp\bigl(n\log(eN/n)-(N-n)\cdf(s_n)\bigr) \\
        &= 2 \exp\bigl(n\log(N)- N \cdf(s_n) \bigr) \exp \Bigl( n \Bigl( \log \Bigl( \frac{e}{n} \Bigr) + \cdf(s_n) \Bigr)  \Bigr) .
    \end{align*}	
    Since, for $ n \geq e^2 $, we have $ \log \bigl( \frac{e}{n} \bigr) + \cdf(s_n) \leq 0$, we get
    \begin{align*}
    \Pr{s_nB_2^n \nsubseteq P_{N,n}^\beta} 
    \leq 2 \exp\bigl(n\log(N)- N \cdf(s_n) \bigr)
    \to 0 .
    \end{align*}

    \bigskip
    (b)
    Set $ r_n = ( 1 - \epsilon ) \sqrt{ n } $, $ A_n = \R^n \setminus r_n B_2^n $ and assume $ N \tilde{ \cdf } ( r_n ) \to 0 $.
    By the thin shell property of $ \mu $, see Theorem \ref{thm:concentration}, we have that
    $ \Ex  \mu ( \tilde{P}_{N,n}^{\beta,\sigma} \cap r_n B_2^n ) \to 0 $.
    On the other hand Lemma \ref{lemma.ineq} (b) gives that 
    \[ 
        \Ex  \mu ( \tilde{P}_{N,n}^{\beta,\sigma} \cap A_n ) 
        \leq N \sup_{ x \in A_n } \tilde{ \cdf } ( \norm{ x }_2 ) 
        = N \tilde{ \cdf } ( r_n ) 
        \to 0 .
    \]
    Therefore
    \[ 
        \Ex  \mu ( \tilde{P}_{N,n}^{\beta,\sigma} ) 
        = \Ex  \mu ( \tilde{P}_{N,n}^{\beta,\sigma} \cap r_n B_2^n ) + \Ex  \mu ( \tilde{P}_{N,n}^{\beta,\sigma} \cap A_n ) 
        \to 0 .
    \]
    
    \medskip
    Now, set $ s_n = ( 1 + \epsilon ) \sqrt{ n } $ and assume that $ N \tilde{ \cdf } ( s_n ) - n \log N \to \infty $.
    From the lower bound in \Cref{lemma.ineq} (b) we get that
    \[
    \Ex \mu(\tilde{P}_{N,n}^{\beta,\sigma}) \gr\mu( s_n B_2^n)\Pr{s_n B_2^n\subseteq \tilde{P}_{N,n}^{\beta,\sigma}}. 
    \]
    On one hand the thin shell property of $ \mu $, see Theorem \ref{thm:concentration}, gives that
    $ \Ex  \mu ( s_n B_2^n ) \to 1 $.
    On the other hand, arguing exactly as in the proof of case (a), we can use the bound
    \[
    \Pr{RB_2^n \nsubseteq \tilde{P}_{N,n}^{\beta,\sigma}}
    \ls 2\exp\left(n\log N - N\tilde{\cdf}(R) \right)
    \to 0 .
    \]
    Therefore
    \[
     1
     \geq \Ex \mu(\tilde{P}_{N,n}^{\beta,\sigma}) 
     \geq \mu( s_n B_2^n) ( 1 - \Pr{s_n B_2^n\nsubseteq \tilde{P}_{N,n}^{\beta,\sigma}} )
     \to 1, 
    \]
    which completes the proof.
\end{proof}

\subsection{Proofs regarding the beta model}

Based on these preparations, we proceed to the proof of Theorem \ref{Theorem1} on the volume of beta polytopes. 

\noindent\textbf{Proof of Theorem \ref{Theorem1}:} Set $r_n=\sqrt{1-n^{-(1-\frac{\epsilon}{2})}}$.
From Lemma \ref{lem:B.bounds} we get
\begin{equation*}
\begin{split}
\cdf(r_n) &\ls \frac{n^{-(1-\frac{\epsilon}{2})(\beta+\frac{n+1}{2})}}{\sqrt{\beta+\frac{n}{2}}}\\
          &= \exp\Bigl(-\Bigl(1-\frac{\epsilon}{2}\Bigr)\Bigl(\beta+\frac{n+1}{2}\Bigr)\log n - \frac{1}{2}\log\Bigr(\beta+\frac{n}{2}\Bigr) \Bigr).
\end{split}
\end{equation*}
The choice $N\ls \exp\bigl((1-\epsilon)\bigl(\beta+\frac{n+1}{2}\bigr)\log n\bigr)$ implies that
\[
N\cdf(r_n) \ls \exp\Bigl(-\frac{\epsilon}{2}\Bigl(\beta+\frac{n+1}{2}\Bigr)\log n - \frac{1}{2}\log\left(\beta+\frac{n}{2}\right) \Bigr)\to 0,
\]
as \(n\to\infty\).
Combined with Lemma \ref{lemma.beta.prime.productNF}, this yields the proof of the first part of the theorem.

\bigskip

Set  $R_n=\sqrt{1-n^{-(1+\frac{\epsilon}{2})}}$.
From Lemma \ref{lem:B.bounds} we get
\begin{align*}
  \cdf(R_n) 
  &\gr \frac{1}{2\sqrt{\pi}} \frac{ n^{-(1+\frac{\epsilon}{2}) (\beta + \frac{n+1}{2}) } }{ \sqrt{ \beta + \frac{n}{2} + 1 } } \\
  &= \exp\Bigl(-\Bigl(1+\frac{\epsilon}{2}\Bigr) \Bigl(\beta+\frac{n+1}{2} \Bigr)\log n-\frac{1}{2}\log\Bigl(4\pi\left(\beta+\frac{n}{2}+1\right)\Bigr)\Bigr).
\end{align*}
The choice $N = \exp\left((1+\epsilon)(\beta+\frac{n+1}{2})\log n\right)$ implies that
\begin{align*}
N\cdf(R_n) &\gr \exp\Bigl(\frac{\epsilon}{2} \Bigl(\beta+\frac{n+1}{2} \Bigr)\log n - \frac{1}{2}\log\Bigl(4\pi\Bigl(\beta+\frac{n}{2}+1\Bigr) \Bigr)\Bigr) ,
\end{align*}
and thus
\[
\lim_{n\to \infty} NF(R_n) - n\log N  = \infty .
\]
Combined with Lemma \ref{lemma.beta.prime.productNF}, this yields the proof.
\qed 

	\begin{remark}
	\label{remark:general}
	As anticipated in Section \(1\), \Cref{thm:beta} can be formulated in a stronger way, as follows. 
	
	Consider any function \(f=f(n)\) such that \(f(n)\to\infty\) and \(f(n)-\log n\to -\infty\) as \(n\to\infty\). If \(N\ls\exp\bigl((\beta+\frac{n+1}{2})f(n)\bigr)\), then, \(\Ex \V{P^\beta_{N,n}}/\vball_n\to0\) as \(n\to\infty\). Analogously, for any function \(g=g(n)\) such that \(g(n)-
	\log n\to +\infty\) as \(n\to\infty\), if \(N\gr\exp\bigl((\beta+\frac{n+1}{2})g(n)\bigr)\), then, \(\Ex \V{P_{N,n}^\beta}/\vball_n\to1\) as \(n\to\infty\). 
	This is proved in the same way as \Cref{thm:beta} using \(r_n^2=1-\exp(-f(n)/2)\) for the upper bound and \(R_n^2=1-\exp(-g(n)/2)\) for the lower bound, respectively.
	
	Notice that this is equivalent to replacing \(\epsilon\) constant in the statement by  \(\epsilon=\epsilon(n)\) with \(\epsilon(n)\gg 1/\log n\).
	
	This formulation also clarifies the fact that the only regimes that remain uncovered by \Cref{thm:beta} are \(N\simeq\exp\bigl((\beta+\frac{n+1}{2})\log (c\;\!n)\bigr)\) for any constant \(c>0\).
	
    	\end{remark}

\noindent\textbf{Proof of Corollary \ref{cor.sphere}:} We start with the first case. Let $\epsilon\in(0,1)$ and fix a sequence $N(n)\leq \exp((1-\varepsilon)(\frac{n-1}{2})\log n)$. As elaborated in \cite{GKT} the weak limit of a sequence of beta distributions on $\R^n$ for $\beta\to -1$ is the unique rotational invariant probability measure on the sphere $S^{n-1}$, for any fixed $n$. Since the map $(x_1,\dots,x_N)\mapsto V_n(\mathrm{conv}(x_1,\dots,x_N))/V_n(B_2^n)$ is bounded and continuous, there exists a sequence $\beta_n$ such that $| \Ex  V_n(P_{N,n}^{\beta_n}) - \Ex  V_n(S_{N,n}) |<\varepsilon^\prime V_n(B_2^n)$, for any $\varepsilon^\prime>0$. By Theorem \ref{Theorem1} we have $\Ex  V_n(P_{N,n}^\beta) \leq \varepsilon^\prime V_n(B_2^n)$, and thus can conclude that $\Ex  V_n(S_{N,n}) \leq 2 \varepsilon^\prime V_n(B_2^n)$. The statement of the second case can be shown analogously. \qed

\subsection{Intrinsic volumes of the beta polytopes} \label{sssec.intrinsic}	

\subsubsection*{Dimension reduction}

For the beta distribution, there is a known formula that relates the expected $k$-th intrinsic volume of $P_{N,n}^\beta$ to the expected volume of the respective $k$-dimensional polytope up to a different parameter $\beta'$. In particular, Proposition 2.3 in \cite{KTT} states that
	\[
	\Ex  V_k(P_{N,n}^\beta) = \binom{n}{k}\frac{\vball_n}{\vball_k \vball_{n-k}} \Ex  V_k\big( P_{N,k}^{\beta+\frac{n-k}{2}} \big).
	\]
	Since 
	\begin{align*}
	V_k(B_2^n)=\binom{n}{k}\frac{\vball_n}{\vball_k \vball_{n-k}} V_k(B_2^k), 
	\end{align*}
	see, e.g., Equation (4.8) from \cite{SchneiderBuch}, we have
\begin{equation}\label{eq:beta.intrinsic.identity}
\frac{\Ex  V_k(P_{N,n}^\beta)}{V_k(B_2^n)} = \frac{\Ex  V_k\big( P_{N,k}^{\beta+\frac{n-k}{2}} \big)}{V_k(B_2^k)}.
\end{equation}

The above relation indicates that for any $k=k(n)$ such that $\lim_{n\to\infty} k(n)=\infty$, a threshold behavior similar to that of Theorem \ref{Theorem1} holds for the intrinsic volumes of $P_{N,n}^\beta$, namely, as $n \to \infty$,
\[
\lim_{n\to\infty}\frac{\Ex  V_k(P_{N,n}^\beta)}{V_k(B_2^n)} =
	\begin{cases} 
	0 & \text{ if } N\ls \exp\left((1-\epsilon)(\beta+\frac{n+1}{2})\log k\right) \\ 
	1 & \text{ if } N\gr \exp\left((1+\epsilon)(\beta+\frac{n+1}{2})\log k\right).
	\end{cases}
\]
Moreover, if $k=n-m$ for any fixed $m\in\mathbb{N}$, the ratio on the left hand side will exhibit a threshold behavior similar to that of the case $k=n$. As a special case, for $m=1$, one can deduce by Theorem \ref{Theorem1} the following threshold phenomenon for the surface area $S_{n-1}$ of $P_{N,n}^\beta$.
\begin{proposition}
Let $\epsilon\in(0,1)$. Then, as $n\to \infty$,
\[
\lim_{n\to\infty}\frac{\Ex (S_{n-1}(P_{N,n}^\beta))}{S_{n-1}(B_2^n)} =
	\begin{cases} 
	0 & \text{ if } N\ls \exp\left((1-\epsilon)(\beta+\frac{n+1}{2})\log n\right) \\ 
	1 & \text{ if } N\gr \exp\left((1+\epsilon)(\beta+\frac{n+1}{2})\log n\right).
	\end{cases}
\]
\end{proposition}

Still, by \eqref{eq:beta.intrinsic.identity}, determining the threshold behaviour of the $k$-th intrinsic volume when $k$ is a fixed integer would require looking into the case that the space dimension stays fixed, while the parameter $\beta$ grows to infinity. This is done in the next subsection.

\subsubsection*{A detour into thresholds for beta polytopes in fixed dimension}
Here we present the proof of Theorem \ref{thm.V_k.fixed_k}. This will come as a corollary of the following general statement.
\begin{theorem}\label{thm.fixed.dim.general}
Let $n\in\mathbb{N}$ be a fixed integer, $\delta > 1$ and $N = \delta^\beta$.
\begin{enumerate}[label={\emph{(\alph*)}}]
\item For any $ R \in \Bigl( 0 , \sqrt{ \frac{ \delta - 1 }{ \delta } } \Bigr) $, we have that $ \Pr{ R B_2^n \subset P_N^\beta } \to 1 $ as $ \beta \to \infty $.
\item For any $ R \in \Bigl( \sqrt{ \frac{ \delta - 1 }{ \delta } } , 1 \Bigr) $, we have that $ \Pr{ P_N^\beta \subset R B_2^n  } \to 1 $ as $ \beta \to \infty $.
\end{enumerate}
\end{theorem}
Given Theorem \ref{thm.fixed.dim.general}, note that if $N=\delta^\beta=\exp(\beta\log\delta)$ and $R_1, R_2$ are such that $0<R_1<\sqrt{\frac{\delta-1}{\delta}}<R_2<1$, then
\[
\lim_{\beta\to \infty}\mathbb{P}(R_1B_2^n\subseteq P_{N,n}^\beta \subseteq R_2B_2^n) = 1.
\]
In particular
\[
\lim_{\beta\to \infty}\mathbb{P}\Bigl(R_1^n\ls \frac{V_n(P_{N,n}^\beta)}{V_n(B_2^n)} \ls R_2^n\Bigr) = 1,
\]
and since this holds for any $0<R_1<\sqrt{\frac{\delta-1}{\delta}}<R_2<1$, we get that
\[
\lim_{\beta\to \infty} \frac{\Ex V_n(P_{N,n}^\beta)}{V_n(B_2^n)} = \Bigl( \frac{\delta-1}{\delta} \Bigr)^{\frac{n}{2}}.
\]
Now since $N\mapsto \frac{\Ex V_n(P_{N,n}^\beta)}{V_n(B_2^n)}$ is an increasing function, $\lim_{\delta\to \infty}\left( \frac{\delta-1}{\delta} \right)^{\frac{n}{2}} = 1$ and $\lim_{\delta\to 1}\left( \frac{\delta-1}{\delta} \right)^{\frac{n}{2}} = 0$, we have just proved the following.
\begin{corollary}\label{corol.intrinsic.fixed.dim}
Let $ n \in \N $ be a fixed integer. Let $ f , g  \colon ( -1 , \infty ) \to \R_+ $ be functions with $ f ( \beta ) \to \infty $ and $ g ( \beta ) \to 0 $ as $ \beta \to \infty $, and let $ \delta \in ( 1 , \infty ) $. Then, as $ \beta \to \infty $,
\begin{equation*}
\lim_{n\to\infty} \frac{ \Ex  V_n ( P_{ N , n }^\beta ) }{ V_n ( B_2^n ) } =
\begin{cases}
1 &\text{ if }N \geq \exp ( \beta f ( \beta ) ) \\ 
0 &\text{ if }N \leq \exp ( \beta g ( \beta ) )\\
( \frac{ \delta - 1 }{ \delta })^{ \frac{ n }{ 2 } }& \text{ if }N = \exp ( \beta \log( \delta ) ).
\end{cases}
\end{equation*}
\end{corollary}

\noindent\textbf{Proof of Theorem \ref{thm.V_k.fixed_k}:} The result is an immediate consequence of \eqref{eq:beta.intrinsic.identity} and Corollary \ref{corol.intrinsic.fixed.dim}, with $f\left(\beta+\frac{n-k}{2}\right) = (\beta+\frac{n-k}{2})^\epsilon$ and $g\left(\beta+\frac{n-k}{2}\right) = (\beta+\frac{n-k}{2})^{-\epsilon}$.
\qed

\medskip

It remains to prove Theorem \ref{thm.fixed.dim.general}.

\begin{proof}[Proof of Theorem \ref{thm.fixed.dim.general}]
(a) By Lemma \ref{lem:B.bounds} we have that
\[
F ( R )
    	    \geq \frac{ 1 }{ 2 \sqrt{ \pi } } \frac{ ( 1 - R^2 )^{ \beta + \frac{ n + 1 }{ 2 } } }{ \sqrt{ \beta + \frac{ n }{ 2 } + 1 } },
\]
thus
\[
N F ( R ) \geq \frac{ ( 1 - R^2 )^{ \frac{ n + 1 }{ 2 } } }{ 2 \sqrt{ \pi } } \frac{ 1 }{ \sqrt{ \beta + \frac{ n }{ 2 } + 1 } } ( \delta ( 1 - R^2 ) )^{ \beta } .
\]
Observe that $ \epsilon := \delta ( 1 - R^2 ) - 1 > 0 $ because $ R < \sqrt{ \frac{ \delta - 1 }{ \delta } } $. It is then easy to see that
\[
\lim_{\beta\to\infty}\frac{ ( 1 - R^2 )^{ \frac{ n + 1 }{ 2 } } }{ 2 \sqrt{ \pi } } \frac{ 1 }{ \sqrt{ \beta + \frac{ n }{ 2 } + 1 } } ( 1 + \epsilon )^{ \frac{ \beta }{ 2 } } 
            = +\infty,
\]
in particular $NF(R)\gr (1+\varepsilon)^{\beta/2}$ for large enough $\beta$. On the other hand, by Lemma \ref{lem.prob} (a),
\begin{equation*}
            \begin{split}
                1 - \Pr{ R B_2^n \subset P_N^\beta } 
                & \leq 2 \binom{ N }{ n } ( 1 - F ( R ) )^{  N - n } 
                \\ & \leq 2 N^n ( 1 - F ( R ) )^{  N - n } 
                \\ & = \exp \left(  \log ( 2 ) + n \log ( N )  + ( N - n ) \log ( 1 - F ( R ) )  \right)
                \\ & \leq \exp \left( \log ( 2 ) + n \log ( N ) - ( N - n ) F ( R )  \right),
            \end{split}
        \end{equation*}
and since $\log N = \beta\log \delta$ and $n$ is fixed, we have that the last expression tends to $0$ as $\beta\to \infty$. Thus,
\[
\lim_{\beta\to\infty} \mathbb{P}(RB_2^n\subseteq P_{N,n}^\beta) = 1.
\]

\noindent(b) Using integration in polar coordinates and the change of variables $s=t^2$, we can see that if $x$ is distributed according to $\nu_\beta$ one has
\begin{equation*}
\begin{split}
\mathbb{P}(\norm{x}\gr R) &= c_{n,\beta}\int_{(RB_2^n)^c} (1-\norm{x}^2)^\beta\de x\\
                       &= n c_{n,\beta} \kappa_n \int_R^1 (1-t^2)t^{n-1}\de t\\
                       &= \frac{1}{B\left(\beta+1,\frac{n}{2}\right)} \int_{R^2}^1 (1-s)^\beta s^{\frac{n}{2}-1}\de s\\
                       &\ls \frac{1}{B\left(\beta+1,\frac{n}{2}\right)} \int_{R^2}^1 (1-s)^\beta \de s = \frac{(1-R^2)^{\beta+1}}{B\left(\beta+1,\frac{n}{2}\right)(\beta+1)},
\end{split}                     
\end{equation*}
where $B(a,b) = \frac{\Gamma(a)\Gamma(b)}{\Gamma(a+b)}$ denotes the Beta function. Letting $N=\delta^\beta$ and $\varepsilon := 1-\delta(1-R^2)$, the above inequality implies that
\[
N\mathbb{P}(\norm{x}\gr R) \ls (1-\varepsilon)^\beta \frac{1-R^2}{B\left(\beta+1,\frac{n}{2}\right)(\beta+1)}.
\]
Note that $\varepsilon\in(0,1)$, since $R\in \Bigl(\sqrt{\frac{\delta-1}{\delta}}, 1 \Bigr)$, so using the fact that $B\left(\beta+1,\frac{n}{2}\right) \sim \Gamma(n/2)/(\beta+1)^{n/2}$ we can easily see that
\[
\lim_{\beta\to \infty} (1-\varepsilon)^{\frac{\beta}{2}} \frac{1-R^2}{B\left(\beta+1,\frac{n}{2}\right)(\beta+1)} = 0.
\]
In particular, $N\mathbb{P}(\norm{x}\gr R) \ls (1-\varepsilon)^{\frac{\beta}{2}}$ if $\beta$ is sufficiently large. Combining this with the inequality $\log x \gr1-\frac{1}{x}$, which holds for every $x>0$, we get
\begin{align*}
0\gr N\log\mathbb{P}(\norm{x}\ls R)&\gr N\Bigl(1-\frac{1}{\mathbb{P}(\norm{x}\ls R)} \Bigr) \\
                                &\gr N \left(1-\frac{1}{1-\frac{(1-\varepsilon)^{\beta/2}}{N}} \right) = -\frac{N(1-\varepsilon)^{\beta/2}}{N-(1-\varepsilon)^{\beta/2}}.
\end{align*}
It follows that $\lim_{\beta\to \infty} N\log\mathbb{P}(\norm{x}\ls R) = 0$. By independence, we have that
\[
\mathbb{P}(P_{N,n}^\beta\subseteq RB_2^n) = \mathbb{P}(\norm{x}\ls R)^N = \exp(N\log\mathbb{P}(\norm{x}\ls R)),
\]
which gives that $\lim_{\beta\to \infty}\mathbb{P}(P_{N,n}^\beta\subseteq RB_2^n) = 1$, completing the proof.
\end{proof}

\subsection{Proofs regarding the beta-prime model}
  
 Using Lemma \ref{lem.tail.Ftilda} and the machinery developed in Section \ref{subsec.lemmas}, we now proceed to the proof of Theorem \ref{thm.beta.prime.sigma}. Set 
  \[
    b_n = \beta - \frac{ n - 1 }{ 2 } .
  \]
  Under the assumptions of Theorem \ref{thm.beta.prime.sigma}, $b_n\to\infty$.
  Thus \eqref{eq.cdf.rewritten.aux} becomes
    \begin{equation} \label{eq.cdf.rewritten}
        \tilde{ \cdf } ( d ) 
        \sim \frac{ 1 }{ \sqrt{ 2 \pi } } \int_{ a_n }^\infty \left( 1 + \frac{ s^2 }{ 2 b_n } \right)^{ - b_n } \de{ s },
        \quad  a_n =  d \frac{ \sqrt{ 2 b_n } }{ \sigma } .
    \end{equation}
    
    \noindent\textbf{Proof of Theorem \ref{thm.beta.prime.sigma} (a): }
      Let $ \epsilon > 0 $.
      Equation \eqref{eq.cdf.rewritten} gives that
      \begin{equation*}
          \tilde{ \cdf } \big( ( 1 + \epsilon ) \sqrt{ n } \big) 
          \sim \frac{ 1 }{ \sqrt{ 2 \pi } } \int_{ a_n }^\infty \left( 1 + \frac{ s^2 }{ 2 b_n } \right)^{ - b_n } \de{ s },
      \end{equation*}
      with
      $ \frac{ a_n^4 }{ b_n } 
      = 4 ( 1 + \epsilon )^4 \frac{ n^2 ( \beta - \frac{ n - 1 }{ 2 } ) }{ \sigma^4 } 
      \to 0 $
      because of the assumptions.
      Thus, by Lemma \ref{lem.tail.Ftilda},
      \[ \tilde{ \cdf } \big( ( 1 + \epsilon ) \sqrt{ n } \big) 
      \sim \frac{ 1 }{ \sqrt{ 2 \pi } } \int_{ a_n }^{ \infty } e^{ - \frac{ t^2 }{ 2 } } \de{ t } . \]
      Since $ a_n = ( 1 + \epsilon ) \frac{ \sqrt{ 2 b_n n } }{ \sigma } \to 0 $,
      it follows that 
      \[ \tilde{ \cdf } \big( ( 1 + \epsilon ) \sqrt{ n } \big) 
      \to \frac{ 1 }{ 2 } . \] 
      Therefore, for $ N =3\lceil  n \log n \rceil$ and $ n $ big enough, we have
      \begin{equation*}
          \begin{split}
              N \tilde{ \cdf } \big( ( 1 + \epsilon ) \sqrt{ n } \big) - n \log N 
              & \geq \frac{ 2 }{ 5 } N- n \log N
              \\ & = \frac{ 6 }{ 5 } \lceil n \log n \rceil - n \log (3\lceil n\log n\rceil)
              \to \infty  .
          \end{split}
      \end{equation*}
      Again, Lemma \ref{lemma.beta.prime.productNF} yields the proof. \qed
    
    \medskip
    \noindent\textbf{Proof of Theorem \ref{thm.beta.prime.sigma} (b):} 
      From \eqref{eq.cdf.rewritten}, we have
      \begin{equation*}
          \tilde{ \cdf } \big( ( 1 - \epsilon ) \sqrt{ n } \big) 
          \sim \frac{ 1 }{ \sqrt{ 2 \pi } } \int_{ a_n }^\infty \left( 1 + \frac{ s^2 }{ 2 b_n } \right)^{ - b_n } \de{ s }.
      \end{equation*}
      Due to the assumptions, $ a_n = ( 1 - \epsilon ) \frac{ \sqrt{ 2 b_n n } }{ \sigma } \to \infty $ and $ \frac{ a_n^4 }{ b_n } 
      = 4 ( 1 - \epsilon )^4 \frac{ n^2 ( \beta - \frac{ n - 1 }{ 2 } ) }{ \sigma^4 } 
      \to 0 $.
      
      Thus, by Lemma \ref{lem.tail.Ftilda},
      \[ \tilde{ \cdf } \bigl( ( 1 - \epsilon ) \sqrt{ n } \bigr) 
      \sim \frac{ e^{ - \frac{ a_n^2 }{ 2 } } }{ \sqrt{ 2 \pi } a_n }
      = \frac{ 1 }{ \sqrt{ 2 \pi } a_n } \exp \Bigl( - ( 1 - \epsilon )^2 \frac{ b_n n }{ \sigma^2 } \Bigr) . \]
      In particular, for 
      $ N \leq \exp\left( ( 1 - \epsilon )^2 \frac{ n b_n }{ \sigma^2 } \right) $ 
      and $ n $ big enough,
      \[ N \tilde{ \cdf } \big( ( 1 - \epsilon ) \sqrt{ n } \big) 
      \leq \frac{ 1 }{ a_n }
      \to 0 , \]
      which implies
      \[ \Ex \mu(\tilde{P}_{N,n}^{\beta,\sigma}) \to 0 , \]
      because of Lemma \ref{lemma.beta.prime.productNF}.
      
      \bigskip
      Similarly as above, we have 
      \[ \tilde{ \cdf } \bigl( ( 1 + \epsilon ) \sqrt{ n } \bigr) 
      \sim \frac{ 1 }{ \sqrt{ 2 \pi } a_n } \exp \Bigl( - ( 1 + \epsilon )^2 \frac{ b_n n }{ \sigma^2 } \Bigr) , \]
      where $ a_n = ( 1 + \epsilon ) \frac{ \sqrt{ 2 b_n n } }{ \sigma } $.
      Because of the condition $ \frac{ b_n n }{ \sigma^2 } \to \infty $, we have that, for $ n $ big enough,
      \begin{equation*}
          \begin{split}
              \tilde{ \cdf } \bigl( ( 1 + \epsilon ) \sqrt{ n } \bigr) 
              & \sim \exp \Bigl( - ( 1 + \epsilon )^2 \frac{ b_n n }{ \sigma^2 } - \frac{ 1 }{ 2 } \log \frac{ b_n n }{ \sigma^2 } - \log \big( 2 ( 1 + \epsilon ) \sqrt{ \pi } \big) \Bigr)
              \\ & \geq \exp \Bigl( - ( 1 + 3 \epsilon ) \frac{ b_n n }{ \sigma^2 } \Bigr) ,
          \end{split}
      \end{equation*}
      where the inequality holds because $ ( 1 + \epsilon )^2 < 1 + 3 \epsilon $.
      Hence, for $ N = \exp \left( ( 1 + 4 \epsilon ) \frac{ b_n n }{ \sigma^2 } \right) $ and $n$ big enough, we have
      \begin{equation} \label{eq.LB.inproof}
          \begin{split}
              N \tilde{ \cdf } \big( ( 1 + \epsilon ) \sqrt{ n } \big) - n \log N
              & \geq \exp \left( \epsilon % \epsilon' 
              \frac{ b_n n }{ \sigma^2 } \right) - n ( 1 + 4 \epsilon ) % ( 1 + 2 \epsilon' ) 
              \frac{ b_n n }{ \sigma^2 }
              \\ & % \geq \exp \left( f ( n ) \right) - 2 n f ( n ) ,
              \geq \exp \left( f ( n ) \right) - \frac{ 1 + 4 \epsilon }{ \epsilon } n f ( n ) ,
          \end{split}
      \end{equation}
      where $ f ( n ) := \epsilon
      \frac{ b_n n }{ \sigma^2 } $.
      The assumption on the growth of \(\beta\), together with \eqref{eq.LB.inproof}, give that 
      $ N \tilde{ \cdf } \big( ( 1 + \epsilon ) \sqrt{ n } \big) - n \log N 
      \to \infty $,
      and Lemma \ref{lemma.beta.prime.productNF} yields the proof.
    \qed
    
    \medskip
    
    \noindent\textbf{Proof of Theorem \ref{thm.beta.prime.sigma} (c):}
      From \eqref{eq.cdf.rewritten} we have
      \begin{equation*}
          \tilde{ \cdf } \big( ( 1 - \epsilon ) \sqrt{ n } \big) 
          \sim \frac{ \tilde{\alpha}_{n,\beta} }{ \sqrt{ 2 b_n } } \int_{ a_n }^\infty \left( 1 + \frac{ s^2 }{ 2 b_n } \right)^{ - b_n } \de{ s },
      \end{equation*}
      where $ a_n = ( 1 - \epsilon ) \frac{ \sqrt{ 2 b_n n } }{ \sigma } $.
      Note that
      $ \frac{ a_n^2 }{ b_n } 
      = ( 1 - \epsilon )^2 \frac{ 2 n }{ \sigma^2 }
      \to \infty $
      because of the assumption $ \frac{ n }{ \sigma^2 } \to \infty $.
      Consequently, by Lemma \ref{lem.tail.Ftilda}
      \begin{equation*}
          \begin{split}
              \tilde{ \cdf } \big( ( 1 - \epsilon ) \sqrt{ n } \big) 
              & \sim \frac{ 1 }{ \sqrt{ 2 \pi } } \frac{ \sqrt{ b_n } }{ \sqrt{ 2 } ( b_n - \frac{ 1 }{ 2 } ) } \left( 1 + \frac{ a_n^2 }{ 2 b_n } \right)^{ - ( b_n - \frac{ 1 }{ 2 } ) }
              \\ & \sim \frac{ 1 }{ 2 \sqrt{ b_n  \pi } } \left( 1 + ( 1 - \epsilon )^2 \frac{ n }{ \sigma^2 } \right)^{ - ( b_n - \frac{ 1 }{ 2 } ) } .
          \end{split}
      \end{equation*}
      % Note also that $ \frac{ 1 }{ 2 \sqrt{ b_n \pi } } < \frac{ 1 }{ 2 } $ since $ b_n > \frac{ 1 }{ 2 } $.
      In particular, for 
      $ N \leq \exp\left( ( b_n - \frac{ 1 }{ 2 } ) \log \left( ( 1 - \epsilon )^2 \frac{ n }{ \sigma^2 } \right) \right) $ and $ n $ big enough,
      \[ N \tilde{ \cdf } \big( ( 1 - \epsilon ) \sqrt{ n } \big) 
      % \leq \exp \left( - \left( b_n - \frac{ 1 }{ 2 } \right) \log \frac{ 1 + ( 1 - \epsilon )^2 \frac{ n }{ \sigma^2 } }{ ( 1 - 2 \epsilon )^2 \frac{ n }{ \sigma^2 } } \right)
      % \leq \exp \left( - \left( b_n - \frac{ 1 }{ 2 } \right) \log \frac{ ( 1 - \epsilon )^2 }{ ( 1 - 2 \epsilon )^2 } \right)
      \leq \frac{ 1 }{ \sqrt{ b_n } }
      \to 0 , \]
      which implies
      \[ \Ex \mu(\tilde{P}_{N,n}^{\beta,\sigma}) \to 0 , \]
      because of Lemma \ref{lemma.beta.prime.productNF}.
      
      \bigskip
      Similarly as above, we have
      \begin{equation*}
          \begin{split}
              \tilde{ \cdf } \big( ( 1 + \epsilon ) \sqrt{ n } \big) 
              & \sim \frac{ 1 }{ 2 \sqrt{ b_n  \pi } } \left( 1 + ( 1 + \epsilon )^2 \frac{ n }{ \sigma^2 } \right)^{ - ( b_n - \frac{ 1 }{ 2 } ) } .
          \end{split}
      \end{equation*}
      Set $ \epsilon' \in \left( 0 , \log \frac{ ( 1 + 2 \epsilon )^2 }{ ( 1 + \epsilon )^2 }\right) $.
      From the last equation, it is easy to see that for
      $ N = \exp\left( ( b_n - \frac{ 1 }{ 2 } ) \log \left( ( 1 + 2 \epsilon )^2 \frac{ n }{ \sigma^2 } \right) \right) $,
      and $ n $ big enough,
      \begin{equation*}
          \begin{split}
              N \tilde{ \cdf } \big( ( 1 + \epsilon ) \sqrt{ n } \big)
              & \geq \frac{ 1 }{ 4 \sqrt{ b_n } } \exp \left( \left( b_n - \frac{ 1 }{ 2 } \right) \log \frac{ ( 1 + 2 \epsilon )^2 \frac{ n }{ \sigma^2 } }{ 1 + ( 1 + \epsilon )^2 \frac{ n }{ \sigma^2 } } \right) 
              \\ & \geq \exp \left( \epsilon' b_n \right) .
              % \\ & \geq \exp \left( \epsilon' \left( b_n - \frac{ 1 }{ 2 } \right) \right) .
          \end{split}
      \end{equation*}
      Observe also that $ \log \left( ( 1 + 2 \epsilon )^2 \frac{ n }{ \sigma^2 } \right) < \frac{ n }{ 2 } $ because of the assumption $ \sigma > e^{ - \frac{ n }{ 3 } } $.
      Combined with $ \log n \ll b_n $ we get
      \begin{equation*}
          N \tilde{ \cdf } \big( ( 1 + \epsilon ) \sqrt{ n } \big) - n \log N
          % \geq \exp \left( \epsilon' b_n \right)  - n \left( b_n - \frac{ 1 }{ 2 } \right) \frac{ n }{ 2 }
          \geq \exp \left( \epsilon' b_n \right)  -  \frac{ n^2 }{ 2 } b_n
          \to \infty ,
      \end{equation*}
      and the result follows from Lemma \ref{lemma.beta.prime.productNF}.
   \qed
   
   \medskip
    
\noindent\textbf{Proof of Corollary \ref{cor.Gaussian}:} We will prove the corollary just for the first case, since the second is analogous. Fix \(\epsilon\in(0,1/2)\) and a sequence \(N(n)\le\exp\bigl(\bigl(\frac{1}{2} - \epsilon\bigr) n\bigr)\).
  For any fixed \(n\), it holds that \(\tilde p_{n,\beta,\sigma}(x)\to(2\pi)^{-\frac{n}{2}}\exp\bigl(-\frac{\norm{x}^2}{2}\bigr)\) whenever \(\sigma^2=2\beta\to\infty\).
  In particular, for any \(\varepsilon^\prime>0\), one can find \(\sigma_n\) such that  \(\abs[\big]{\Ex \mu(\tilde P_{N,n}^{\frac{1}{2}\sigma_n^2,\sigma_n})-\Ex \mu(G_{N,n})}<\varepsilon^\prime\). 
  We can choose \(\sigma_n\gg n\) and \(\beta_n=\frac{1}{2}\sigma_n^2\), so that the conditions of case \ref{forGauss} in \Cref{thm.beta.prime.sigma} are met. 
  Therefore, for \(n\) large enough, \(\Ex \mu(\tilde P_{N,n}^{\frac{1}{2}\sigma_n^2,\sigma_n})<\varepsilon^\prime\). We can conclude that \(\Ex \mu(G_{N,n})<2\varepsilon^\prime\), which ends the proof.
  \qed

\section{Intersections of random halfspaces}\label{sec.random.facets}

We now study polytopes generated as intersections of random halfspaces. The model is again based on a choice of $N$ independent random points $X_1,\ldots,X_N$ in $\mathbb{R}^n$, but now we consider the polytope generated as the intersection of the halfspaces $\{x\in \mathbb{R}^n : \langle X_i,x\rangle\ls a \}$ for each  $i=1,\ldots,N$, for some $a>0$. Following Pivovarov \cite{Piv}, who treated this setting for the cases where $X_1,\ldots,X_N$ are chosen according to the uniform measure on $S^{n-1}$ or the Gaussian measure, we provide two threshold results of the same nature when the points $X_i$ are chosen according to $\nu_\beta$ or $\tilde{\nu}_\beta$. For simplicity, we choose to treat only the case $\sigma=1$ in the beta-prime model, so let $\tilde{\nu}_\beta := \tilde{\nu}_{\beta,1}$ for the rest of the text.%; still, one could derive results of the same nature for some general $\sigma=\sigma(n,\beta)$.
We give the exact statements and, then, provide a couple of preparatory lemmas before we proceed to the proofs.

Like before, let $N>n$ and $X_1,\ldots,X_N \in B_2^n$ be distributed independently according to $\nu_\beta$. Consider the polytope
\[
H_{N,n}^\beta := \{x\in\mathbb{R}^n : \langle X_i,x \rangle \ls 1, \hbox{ for each } i=1,\ldots,N\}.
\]
Note that $H_{N,n}^\beta$ contains $B_2^n$. Actually we have the stronger inclusion $H_{N,n}^\beta\supseteq RB_2^n$ for $R=\min_{i\in [N]}\|X_i\|^{-1}$. Moreover, the bigger $N$ is, the smaller $V_n(H_{N,n}^\beta)$ gets. The next statement gives a threshold for the normalized volume of intersections of $H_{N,n}^\beta$  with balls of radius larger than $1$. It turns out that, as $n\to \infty$, $H_{N,n}^\beta$ tends to capture the whole of the volume of such a ball, if the number of points is at most exponential in the dimension. 

\begin{theorem}\label{thm.H_Nb}
Fix $\epsilon, R\in(0,1)$ and let $-1<\beta=\beta(n)$. Then,
\[
    \lim_{n\to\infty}\frac{\Ex  V_n(H_{N,n}^\beta\cap R^{-1}B_2^n)}{V_n(R^{-1}B_2^n)} =
	\begin{cases} 
	1 & \text{ if } N\ls \exp\left((1-\epsilon)(\beta+\frac{n+1}{2})\log((1-R^2)^{-1})\right) \\ 
	0 & \text{ if } N\gr \exp\left((1+\epsilon)(\beta+\frac{n+1}{2})\log((1-R^2)^{-1})\right).
	\end{cases}
\]
\end{theorem}

For $X_1,\ldots,X_N \in \mathbb{R}^n$ independent and distributed according to $\tilde{\nu}_\beta$, let
\[
\tilde{H}_{N,n}^\beta := \{x\in\mathbb{R}^n : \langle X_i,x \rangle \ls n, \hbox{ for each } i=1,\ldots,N\}.
\]
As in the case of $\tilde{P}_{N,n}^\beta$, we give a threshold result for $\mu(\tilde{H}_{N,n}^\beta)$, where $\mu$ can be taken to be any isotropic log-concave probability measure on $\mathbb{R}^n$.

\begin{theorem}\label{thm.H_Nb'}
Fix $\epsilon\in(0,1)$ and let \(\mu = \mu_n \) denote an arbitrary isotropic log-concave measure on $\mathbb{R}^n$. Let  $\beta = \beta(n)>n/2$ such that $\lim_{n\to \infty}\beta-n/2=\infty$. Then,
\[
    \lim_{n\to\infty}\Ex  \mu(\tilde{H}_{N,n}^\beta) =
	\begin{cases} 
	1 & \text{ if } N\ls \exp\left(\left(\beta-\frac{n}{2}\right)\log((1-\varepsilon)n)\right) \\ 
	0 & \text{ if } N\gr \exp\left(\left(\beta-\frac{n}{2}\right)\log((1+\varepsilon)n)\right).
	\end{cases}
\]
\end{theorem}

Towards the proofs, we will once more relate the probability contents of $H_{N,n}^\beta$ and $\tilde{H}_{N,n}^\beta$ to the distribution functions $\cdf$ and $\tilde{\cdf}$ of $\nu_\beta$ and $\tilde{\nu}_\beta$, respectively.

\begin{lemma}\label{lem.Pr=F}
\begin{itemize}
\item[\rm(a)] Let $x\in \mathbb{R}^n\setminus B_2^n$. Then, $\mathbb{P}(x \in H_{N,n}^\beta)= (1-\cdf(\norm{x}^{-1}))^N$.
\item[\rm(b)] Let $x\in \mathbb{R}^n\setminus \{0\}$. Then, $\mathbb{P}(x \in \tilde{H}_{N,n}^\beta)= (1-\tilde{\cdf}(n/\norm{x}))^N$.
\end{itemize}
\end{lemma}

\begin{proof}
Again, we sketch the proof only for the beta case. Using independence and the rotational invariance of $\nu_\beta$, we can write
\begin{align*}
\mathbb{P}(x\in H_{N,n}^\beta) &= \mathbb{P}(\langle X_i,x \rangle\ls 1, \hbox{ for every } i=1,\ldots,N)\\
                               &= \mathbb{P}(\langle X_1,x \rangle\ls 1)^N = \mathbb{P}\Bigl(\langle X_1, \frac{x}{\norm{x}} \rangle\ls \norm{x}^{-1}\Bigr)^N\\
                               &= \mathbb{P}\left(\langle X_1, e_1 \rangle\ls \norm{x}^{-1}\right)^N = (1-\cdf(\norm{x}^{-1}))^N.
\end{align*}
The same argument proves (b).
\end{proof}

\begin{lemma}\label{lem.EVn.bounds}
\begin{itemize}
\item[\rm(a)] Let $1\ls t < s$ and $B=sB_2^n\setminus tB_2^n$. Then,
\[
V_n(B)(1-\cdf(s^{-1}))^N \ls \Ex V_n(H_{N,n}^\beta\cap B) \ls V_n(B)(1-\cdf(t^{-1}))^N.
\]
\item[\rm(b)] Let $\mu$ be an isotropic log-concave measure on $\mathbb{R}^n$, $0<t<s$ and $B=sB_2^n\setminus tB_2^n$. Then,
\[
\mu(B)(1-\tilde{F}(n/s))^N \ls \Ex \mu(\tilde{H}_{N,n}^\beta\cap B)\ls \mu(B)(1-\tilde{\cdf}(n/t))^N. 
\]
\end{itemize}
\end{lemma}

\begin{proof}
Note that, e.g., for (a),
\[
\Ex V_n(H_{N,n}^\beta\cap B) = \Ex  \int_B \mathbbm{1}_{\{x\in H_{N,n}^\beta\}}(x)\,\de x = \int_B \mathbb{P}(x\in H_{N,n}^\beta) \de x,
\]
and the wanted bounds follow from Lemma \ref{lem.Pr=F} (a), and the fact that $x\in B$ is equivalent to $\norm{x}\in(t,s)$.
\end{proof}

\begin{remark}
Using the inequalities
\[
x-1-(x-1)^2 \ls \log x\ls x-1,
\]
that hold for every  $x\in [1/2,1]$, we will apply in the proofs below the bounds
\begin{equation}\label{eq.exp.bd.b}
\exp(-N\cdf(a)-N\cdf(a)^2) \ls (1-\cdf(a))^N \ls \exp(-N\cdf(a))
\end{equation}
for any $a\in (0,1)$, and
\begin{equation}\label{eq.exp.bd.b'}
\exp(-N\tilde{\cdf}(a)-N\tilde{\cdf}(a)^2) \ls (1-\tilde{\cdf}(a))^N \ls \exp(-N\tilde{\cdf}(a))
\end{equation}
for any $a>0$.
\end{remark}

\subsection{Proof of Theorem \ref{thm.H_Nb}}

Let $0<R<1$, $\epsilon\in(0,1)$, and set $t:=(1-(1-R^2)^{1+\frac{\varepsilon}{2}})^{-1/2}$. Hence, we have
\[
\log\sqrt{1-t^{-2}} = \left(1+\frac{\varepsilon}{2} \right)\log\sqrt{1-R^2}.
\]
Now, take $s=R^{-1}$. Note that $1<t<s$, and set $B=sB_2^n\setminus tB_2^n$. Then,
\[
\frac{\Ex V_n(H_{N,n}^\beta\cap sB_2^n)}{V_n(sB_2^n)} = \frac{\Ex V_n(H_{N,n}^\beta\cap tB_2^n)}{V_n(sB_2^n)} + \frac{\Ex V_n(H_{N,n}^\beta\cap B)}{V_n(sB_2^n)}.
\]
In particular,
\begin{equation}\label{eq.bproof.bounds}
\frac{\Ex V_n(H_{N,n}^\beta\cap B)}{V_n(sB_2^n)} \ls \frac{\Ex V_n(H_{N,n}^\beta\cap sB_2^n)}{V_n(sB_2^n)} \ls \Bigl(\frac{t}{s}\Bigr)^n + \frac{\Ex V_n(H_{N,n}^\beta\cap B)}{V_n(sB_2^n)}.
\end{equation}
We will use \eqref{eq.bproof.bounds} to bound the investigated ratio from above and below by something that tends to $1$ or zero respectively, depending on the choice of $N$.

For the lower bound, we let
\[
N\le \exp\Bigl((1-\varepsilon)\Bigl(\beta+\frac{n+1}{2} \Bigr)\log\bigl((1-R^2)^{-1} \bigr) \Bigr).
\]
Then, using successively the lower bounds in \eqref{eq.bproof.bounds}, Lemma \ref{lem.EVn.bounds} (a) and \eqref{eq.exp.bd.b}, we write
\begin{align*}
\frac{\Ex V_n(H_{N,n}^\beta\cap sB_2^n)}{V_n(sB_2^n)} &\gr \frac{\Ex V_n(H_{N,n}^\beta\cap B)}{V_n(sB_2^n)} = \frac{V_n(B)}{V_n(sB_2^n)}\frac{\Ex V_n(H_{N,n}^\beta\cap B)}{V_n(B)} \\
                                                            &\gr (1-(t/s)^n)(1-\cdf(s^{-1}))^N \\
                                                            &\gr (1-(t/s)^n)\exp(-N\cdf(R)-N\cdf(R)^2).
\end{align*}
It thus suffices to prove that $N\cdf(R)\to 0$ and $N\cdf(R)^2\to 0$ to get that $\frac{\Ex V_n(H_{N,n}^\beta\cap sB_2^n)}{V_n(sB_2^n)} \to 1$. Recall that, by Lemma \ref{lem:B.bounds}, we have
\[
\cdf(R) \ls \frac{1}{2R\sqrt{\pi}}\frac{1}{\sqrt{\beta+\frac{n}{2}}}\exp\Bigl(-\Bigl(\beta+\frac{n+1}{2}\Bigr)\log\bigl((1-R^2)^{-1}\bigr)\Bigr),
\]
so that
\[
N\cdf(R) \ls \frac{1}{2R\sqrt{\pi}}\frac{1}{\sqrt{\beta+\frac{n}{2}}}\exp\Bigl(-\epsilon\Bigl(\beta+\frac{n+1}{2}\Bigr)\log\bigl((1-R^2)^{-1}\bigr)\Bigr),
\]
which establishes that $NF(R)\to 0$ as $n\to \infty$. It is straightforward to see that the same holds for $NF(R)^2$.

For the upper bound, we choose
\[
N\ge  \exp\Bigl( (1+\epsilon) \Bigl(\beta+\frac{n+1}{2}\Bigr)\log\bigl((1-R^2)^{-1}\bigr) \Bigr).
\]
Similarly as before, we use the upper bounds in \eqref{eq.bproof.bounds}, Lemma \ref{lem.EVn.bounds} (a) and \eqref{eq.exp.bd.b} to see that
\[
\frac{\Ex V_n(H_{N,n}^\beta\cap sB_2^n)}{V_n(sB_2^n)} \ls \Bigl(\frac{t}{s}\Bigr)^n + \exp(-NF(t^{-1})).
\]
Note that, by the definition of $N$, $t$, and the lower bound of Lemma \ref{lem:B.bounds}, we get
\[
N\cdf(t^{-1}) \gr \frac{1}{2\sqrt{\pi}}\frac{1}{\sqrt{\beta+\frac{n}{2}+1}} \exp\Bigl(\frac{\varepsilon}{2}\Bigl(\beta+\frac{n+1}{2}\Bigr)\log\bigl((1-R^2)^{-1}\bigr)\Bigr),
\]
that yields $\lim_{n\to \infty}N\cdf(t^{-1}) = +\infty$. Since $(t/s)^n\to 0$ as $n\to\infty$, we then have that $\lim_{n\to \infty}\frac{\Ex  V_n(H_{N,n}^\beta\cap sB_2^n)}{V_n(sB_2^n)} = 0$, proving the claim.

\subsection{Proof of Theorem \ref{thm.H_Nb'}}

Let $\epsilon\in(0,1)$ and set $s_n := n/\sqrt{(1-\epsilon/2)n-1}$, $t_n :=n/ \sqrt{(1+\varepsilon/2)n-1}$. Note that \(t_n<s_n\) and both are of order \(\sqrt{n}\). Then, set $B=s_nB_2^n\setminus t_nB_2^n$. Let $\mu$ be an isotropic log-concave probability measure on $\R^n$, and choose
\[
N\gr \exp\left(\left(\beta-\frac{n}{2}\right)\log((1+\varepsilon)n)\right).
\]
Since
\[
\Ex \mu(\tilde{H}_{N,n}^\beta) \ls \mu(t_nB_2^n) + \Ex \mu(\tilde{H}_{N,n}^\beta\cap B) + \mu(\R ^n\setminus s_nB_2^n),
\]
and by Theorem \ref{thm:concentration} the first and last term tend to zero with $n$, we need only to prove that the same happens to the second term.

By the upper bounds in Lemma \ref{lem.EVn.bounds} (b) and \eqref{eq.exp.bd.b'}, we have that
\[
\Ex \mu(\tilde{H}_{N,n}^\beta\cap B) \ls \mu(B)(1-\tilde{\cdf}(n/t_n))^N \ls \mu(B)\exp(-N\tilde{\cdf}(n/t_n)).
\]
By \eqref{eq.tilde(F).sigma=1.bounds}, we have that
\[
N\tilde{\cdf}(n/r_n) \gr \frac{1}{2\sqrt{\pi}}\frac{1}{\sqrt{\beta-\frac{n-1}{2}}} \left(\frac{1+\varepsilon}{1+\frac{\varepsilon}{2}}\right)^{\beta-\frac{n}{2}},
\]
and since the last expression tends to infinity with $n$, we get $\lim_{n\to \infty}\Ex \mu(\tilde{H}_{N,n}^\beta\cap B) = 0$, proving the second statement of the theorem.

On the other hand, let
\[
N\ls \exp\left(\left(\beta-\frac{n}{2}\right)\log((1-\varepsilon)n) \right).
\]
Note that, by \eqref{eq.tilde(F).sigma=1.bounds}, again,
\[
N\tilde{\cdf}(n/s_n) \ls \frac{1}{\sqrt{2\pi}}\frac{1}{\sqrt{\beta-\frac{n+1}{2}}} \left(\frac{1-\varepsilon}{1-\varepsilon/2} \right)^{\beta-\frac{n}{2}},
\]
which tends to zero with $n$. The same holds for $N\tilde{\cdf}(n/s_n)^2$. Using the lower bounds in Lemma \ref{lem.EVn.bounds} (b) and \eqref{eq.exp.bd.b'} we see that
\begin{align*}
\Ex \mu(\tilde{H}_{N,n}^\beta) &\gr\Ex \mu(\tilde{H}_{N,n}^\beta\cap B) \gr \mu(B)(1-\tilde{\cdf}(n/s_n))^N \\
                                     &\gr \mu(B)\exp\bigl(-N\tilde{\cdf}(n/s_n) -N\tilde{\cdf}(n/s_n)^2 \bigr).
\end{align*}
Thus, $\lim_{n\to\infty}\Ex \mu(\tilde{H}_{N,n}^\beta) = 1$, which completes the proof.

\section*{Acknowledgement}
The authors would like to express their gratitude towards Christoph Th\"ale for initiating this collaboration.

	\vspace{1cm}

    \footnotesize

    \noindent\textsc{Gilles Bonnet:} Faculty of Mathematics, Ruhr University Bochum\\
    \textit{E-mail}: \texttt{gilles.bonnet@rub.de}
    
    \bigskip
    
    \noindent\textsc{Giorgos Chasapis:} Department of Mathematics, University of Athens\\
    \textit{E-mail}: \texttt{gchasapis@math.uoa.gr}

    \bigskip

    \noindent\textsc{Julian Grote:} Faculty of Mathematics, Ruhr University Bochum\\
    \textit{E-mail}: \texttt{julian.grote@rub.de}

    \bigskip
    
    \noindent\textsc{Daniel Temesvari:}  Faculty of Mathematics, Ruhr University Bochum\\
    \textit{E-mail}: \texttt{daniel.temesvari@rub.de}

    \bigskip
    
    \noindent\textsc{Nicola Turchi:}  Faculty of Mathematics, Ruhr University Bochum\\
    \textit{E-mail}: \texttt{nicola.turchi@rub.de}
	
\end{document}